\documentclass[a4paper,9pt]{article}
\usepackage[a4paper,bindingoffset=0.2in,%
            left=1in,right=1in,top=1in,bottom=1in,%
            footskip=.25in]{geometry}
\usepackage{graphicx}
\usepackage{amsfonts}
\usepackage{amsmath}
\usepackage{url}
\usepackage{amsthm}
\usepackage{amssymb}
\newtheorem {Proposition}{Proposition}[section]
\newtheorem {Lemma}[Proposition] {Lemma}
\newtheorem {Theorem}[Proposition]{Theorem}
\newtheorem {Corollary}[Proposition]{Corollary}
\newtheorem {Remark}[Proposition]{Remark}

\def\N{\mathbb{N}}

\def\R{\mathbb{R}}
 
\usepackage{enumerate}
\usepackage{bbm}
\usepackage[utf8]{inputenc} % allow utf-8 input
\usepackage[T1]{fontenc}    % use 8-bit T1 fonts
\usepackage{hyperref}       % hyperlinks
\usepackage{url}            % simple URL typesetting
\usepackage{booktabs}       % professional-quality tables
\usepackage{amsfonts}       % blackboard math symbols

\usepackage{nicefrac}       % compact symbols for 1/2, etc.
\usepackage{microtype}      % microtypography
\usepackage{lipsum}
\usepackage{graphicx}
\graphicspath{ {./images/} }

\begin{document}
\title{Central Limit Theorems for Semidiscrete Wasserstein Distances}
\author{Eustasio del Barrio$^{(1)}$\footnote{Research partially supported by FEDER, Spanish Ministerio de Econom\'ia y Competitividad, grant MTM2017-86061-C2-1-P and Junta de Castilla y Le\'on, grants VA005P17 and VA002G18.}, Alberto Gonz\'alez-Sanz$^{(2)}$, and Jean-Michel Loubes $^{(3)}$\footnote{Research partially supported by the AI Interdisciplinary Institute ANITI, which is funded by the French “Investing for the Future – PIA3” program under the Grant agreement ANR-19-PI3A-0004.}\\  $\,$ \\ 
{ $^{(1)(2)}$IMUVA, Universidad de Valladolid, Spain} \\ 
$^{(2)(3)}$IMT, Universit\'e de Toulouse
France\\ $\,$ \\ 
$^{(1)}$tasio@eio.uva.es \quad $^{(2)}$alberto.gonzalez sanz@math.univ-toulouse.fr \\ $^{(3)}$loubes@math.univ-toulouse.fr}
\maketitle
\begin{abstract}
We prove a Central Limit Theorem for the empirical optimal transport cost,  $\sqrt{\frac{nm}{n+m}}\{\mathcal{T}_c(P_n,Q_m)-\mathcal{T}_c(P,Q)\}$, in the semi discrete case, i.e when the distribution $P$ is supported in $N$ points, but without assumptions on $Q$. We show that the asymptotic distribution is the supremun of a centered Gaussian process,  which is Gaussian under some additional conditions on the probability $Q$ and on the cost. Such results imply the central limit theorem for the $p$-Wassertein distance, for $p\geq 1$. 
This means that, for fixed $N$, the curse of dimensionality is avoided. To better understand the influence of such $N$, we provide bounds of $E|\mathcal{W}_1(P,Q_m)-\mathcal{W}_1(P,Q)|$ depending on  $m$  and $N$. 
Finally, the semidiscrete framework provides a control on the second derivative of the dual formulation,  which yields  the first central limit theorem for the optimal transport potentials. The results are supported by simulations that help to visualize the given limits and bounds. We analyse also the cases where classical bootstrap works.
\end{abstract}

%\begin{abstract}

%\end{abstract}

\section{Introduction}

A large number of problems in statistics or computer science require the comparison between histograms or,  more generally, measures. Optimal transport has proven  to be an important tool to compare probability measures since it enables to define a metric over the set of distributions which convey their geometric properties., see \cite{verdinelli}. Moreover, together with the convergence of the moments, it metrizes the weak convergence,  see Chapter 7.1. in \cite{villani2003topics}.  It is nowadays used in a large variety of fields, in probability and  statistics. In particular in Machine learning, OT based methods have been developed to tackle problems in fairness as in \cite{ jiang2020wasserstein,pmlr-v97-gordaliza19a,Barrio2019ACL,delara2021transportbased}, in domain adaptation (\cite{Shen2018WassersteinDG}), or transfer learning (\cite{gayraud2017optimal}). Hence there is a growing need for theoretical results to support such applications and provide theoretical guarantees on the asymptotic distribution.\vskip .1in 

This work focuses on the semi-discrete optimal transport, i.e. when one of both probabilities is supported on a discrete set.  Such a problem is inspired by a large variety of applications, including  resource allocation problem, points versus demand distribution, positions of sites such that the mean allocation cost is minimal (\cite{Hartmann}), resolution of the incompressible Euler equation  using Lagrangian methods (\cite{Gallou}), non-imaging optics; matching between a point cloud and a triangulated surface; seismic imaging( \cite{MEYRON201913}), generation of blue noise distributions with applications for instance to low-level hardware implementation in printers( \cite{deGoes}), in astronomy (\cite{levy2020fast}). From a statistical point of view,  Goodness-of-fit-tests based on semi-discrete optimal transport enable to  detect deviations from a density map to have $P\neq Q$, by using the fluctuations of $\mathcal{W}(P_n,Q)$, see \cite{Hartmann} and to provide a new generalization of  distribution functions and quantile, proposed for instance by \cite{Hallin2020DistributionAQ}, when the probability is discrete.
\vskip .1in
The most general formulation of the optimal transport problem considers $\mathcal{X}, \mathcal{Y}$ both Polish spaces. We use the notation $\mathcal{P}(\mathcal{X})$ (resp. $\mathcal{P}(\mathcal{Y})$) for the set of Borel probability measures on $\mathcal{X}$ (resp. $\mathcal{Y}$). The optimal transport problem  between $P\in \mathcal{P}(\mathcal{X})$ and $Q\in \mathcal{P}(\mathcal{Y})$ for the cost $c:\mathcal{X} \times \mathcal{Y}\rightarrow [0,\infty)$ is formulated as the solution of 
\begin{align}\label{kant}
\mathcal{T}_c(P,Q):=\inf_{\gamma \in \Pi(P,Q)}\int_{\mathcal{X}\times \mathcal{Y}} c(\textbf{x},\textbf{y}) d \pi(\textbf{x}, \textbf{y}),
\end{align}
where $\Pi(P,Q)$ is the set of probability measures $\pi \in \mathcal{P}(\mathcal{X}\times \mathcal{Y})$ such that $\pi(A\times \mathcal{Y})=P(A)$ and $\pi(\mathcal{Y} \times B)=Q(B)$ for all $A,B$ measurable sets. \\ \\
If $c$ is continuous and there exist two continuous functions $a\in L^1(P)$ and $b\in L^1(Q)$ such that 
\begin{equation}\label{cost-cond}
    \text{for all $(\mathbf{x},\mathbf{y})\in \operatorname{supp}(P)\times \operatorname{supp}(Q)$,}\ \ c(\mathbf{x},\mathbf{y})\geq a(\mathbf{x})+b(\mathbf{y}),
\end{equation}
then the Kantorovich problem \eqref{kant} can be formulated in a dual form, as
\begin{align}\label{dual}
\mathcal{T}_c(P,Q)=\sup_{(f,g)\in \Phi_c(P,Q)}\int f(\textbf{x}) dP(\textbf{x})+\int g(\textbf{y}) dQ(\textbf{y}),
\end{align}
where $\Phi_c(P,Q)=\{ (f,g)\in L_1(P)\times L_1(Q): \ f(\textbf{x})+g(\textbf{y})\leq c(\textbf{x},\textbf{y}) \}$, see for instance Theorem 5.10 in \cite{villani2008optimal}. It is said that  $\psi\in L_1(P)$ is an \emph{optimal transport potential from $P$ to $Q$ for the cost $c$} if there exists $\varphi\in L_1(Q)$ such that the pair $(\psi, \varphi)$ solves \eqref{dual}. 
\vspace{5mm} 

We consider observations drawn from  two mutually independent samples $X_1, \dots, X_n$ and $Y_1, \dots, Y_m$  i.i.d. with laws $P$ and $Q$. Let  $P_n=\frac{1}{n}\sum_{k=1}^n\delta_{X_k}$ and $Q_m=\frac{1}{m}\sum_{k=1}^m\delta_{Y_k}$ be the corresponding empirical measures. The optimal transport cost between the empirical distributions $\mathcal{T}_c(P_n,Q_m)$ defines a random variable.
% Supposing that $P,Q\in \mathcal{P}(\R^d)$ and both are absolutely continuous with respect to the Lebesgue measure, \cite{Fournier} proved that $E\mathcal{W}_1(P_n,P)$ converges to $0$ with an order approximately of $\frac{1}{n^{1/d}}$. By using the triangular inequality, valid in Wassertein distances, we can not expect better order of convergence for the difference $\{\mathcal{W}_1(P_n,Q_m)-\mathcal{W}_1(P,Q)\}$.
  The asymptotic distribution of the  empirical transport cost  $\mathcal{T}_c(P_n,Q_m)$ has been studied in some papers. In a very general case, \cite{delbarrio2021central, delbarrio2019,gonzalezdelgado2021twosample} prove, using the Efron-Stein's inequality, a Central Limit Theorem for the centered process,  i.e that  $\sqrt{\frac{nm}{n+m}}\{\mathcal{T}_c(P_n,Q_m)-E\mathcal{T}_c(P,Q)\}$, has a Gaussian asymptotic behavior. With similar arguments, \cite{Weed19} proves that result for the regularized optimal transport cost.  \\
  \indent Under  additional assumptions  it is possible to extend this result, in particular to the semi discrete framework. When $P$ is finitely supported but not $Q$, which  is supposed to be absolutely continuous with respect to the Lebesgue measure, with convex support, and when considering the quadratic cost,   \cite{delbarrio2019} proves that the limit  $\sqrt{n}\{\mathcal{T}_c(P_n,Q)-E\mathcal{T}_c(P,Q)\}$ is in fact Gaussian. Their approach is based on some differentiability properties of the optimal transport problem. But one of their main arguments is that the optimal transport potential is unique which is no longer true for a general costs and neither in general Polish spaces.  Similar results have been proved in the particular semi discrete case of $P,Q$ being supported in a finite (resp. countable) set. In particular, \cite{Sommerfeld2018} (resp. \cite{tameling2019}) prove that in this setting,  $\sqrt{\frac{nm}{n+m}}\{\mathcal{T}_c(P_n,Q_m)-\mathcal{T}_c(P,Q)\}$ has a weak limit $X$, which  is the supremun of a Gaussian process.  
Their proof relies on  the identification of the space of distributions supported in a finite set $\mathbb{X}=\{\mathbf{x}_1, \dots, \mathbf{x}_N\}$ with $\R^N$, and then on a proof based on the directional Hadamard differentiability of the functional $(\mathbf{p},\mathbf{q})\mapsto \mathcal{T}_c(\sum_{i=1}^N p_i\delta_{\mathbf{x}_i}, \sum_{i=1}^N{q_i}\delta_{\mathbf{x}_i})$.   The result of \cite{Sommerfeld2018} establishes that, if $P$ and $Q$ are both supported in a finite set, then
\begin{align*}
	\sqrt{n}\left(\mathcal{T}_c(P_n,Q)- \mathcal{T}_c(P,Q)\right)\overset{w}\longrightarrow \sup_{\mathbf{z}\in \operatorname{Opt}_c^0(P,Q)} \mathbb{G}(\mathbf{z}), \ \ \text{where} \ \ \mathbb{G}(\mathbf{z}):=\sum_{i=1}^N z_i  X_i
\end{align*}
and $(X_1,\dots,X_n)$ is a centered Gaussian vector and $\operatorname{Opt}_c^0(P,Q)$  is the set of solutions of the dual problem \eqref{dual}, both described in section~2.  Lately \cite{tameling2019} extended the same result for probabilities supported in countable spaces.  Yet this approach does not hold for probabilities non supported on finite or countable sets.    \vskip .1in
  In this work we are concerned with the asymptotic behaviour of $\sqrt{\frac{nm}{n+m}}\{\mathcal{T}_c(P_n,Q_m)-\mathcal{T}_c(P,Q)\}$ for general semi discrete setting. We propose a new proof that, even if it still uses the Hadamard differentiability, consider the framework introduced in \cite{Luis2020}. It consists in considering the  Hadamard derivative of the supremun of the process with respect to  $\ell^{\infty}$ topology. The relationship with CLT for optimal transport cost comes from the fact that  the dual formulation of the transport problem is, in fact, a supremum of functions. Hence if such functions lives in a Donsker class (see ~\cite{Vart_Well}), then we can obtain the central limit by proving the   differentiability  of the supremum in $\ell^{\infty}(\mathcal{B})$ and then applying the general delta-method. \vskip .1in
Hence this work first covers and generalizes  previously mentioned results of \cite{Sommerfeld2018} for a semi discrite $P$ approximated by $P_n$ and a general probability distribution $Q$ to handle all cases of  the semi discrete framework.  Moreover, the computation of $\mathcal{T}_c(P_n,Q)$ is not easy in general, see for instance \cite{Gallou}. Consequently, an interesting problem, also for applications, becomes its approximation by a $Q_m$, an estimation of $Q$. Hence we also provide  the asymptotic behaviour of $\sqrt{m}\{\mathcal{T}_c(P,Q_m)-\mathcal{T}_c(P,Q)\}.$  Surprisingly, Theorem~ \ref{Teoremaprinci}  yields that it tends to
$$ \sup_{\mathbf{z}\in \operatorname{Opt}_c(P,Q)} \mathbb{G}_Q(\inf_{i=1, \dots, N} \{ c(\mathbf{x}_i, \mathbf{y} ) -z_i \}),$$
where $\operatorname{Opt}_c(P,Q)$ is the set of optimal transport potentials and $\mathbb{G}_Q$ is the Brownian bridge in $\mathcal{F}_c^K$ (both will be defined more precisely later) with mean zero and covariance 
$$(f,g)\mapsto \int f(\mathbf{y})g(\mathbf{y})dQ(\mathbf{y}) -\int f(\mathbf{y})dQ(\mathbf{y})\int g(\mathbf{y})dQ(\mathbf{y}).$$
Finally we provide in Section~\ref{two_samp} a unified general result that describes  the asymptotic distribution of the empirical  transport cost between a probability $P$ supported in the finite set  $\mathbb{X}=\{ \mathbf{x}_1, \dots, \mathbf{x}_N \}\subset \mathcal{X}$ and  $Q\in \mathcal{P}(\mathcal{Y})$ under the minimal assumption
\begin{align*}
\int c(\mathbf{x}_i, \mathbf{y})dQ(\mathbf{y})<\infty, \ \text{ for all $i=1,\dots, N$},
\end{align*}
for all cases 
\begin{itemize}
    \item \textbf{$P_n \rightarrow P $}
    \begin{align*}
	\sqrt{n}\left(\mathcal{T}_c(P_n,Q)- \mathcal{T}_c(P,Q)\right)\xrightarrow{w} \sup_{\mathbf{z}\in \operatorname{Opt}_c(P,Q)} \mathbb{G}(\mathbf{z}).
\end{align*}
 \item \textbf{$Q_m \rightarrow Q $}, suppose that
\begin{align*}
\int c(\mathbf{x}_i, \mathbf{y})^2dQ(\mathbf{y})<\infty, \ \text{ for all $i=1,\dots,N$},
\end{align*} 
then
$$\sqrt{m}\left(\mathcal{T}_c(P,Q_m)- \mathcal{T}_c(P,Q)\right)\stackrel{w}\longrightarrow \sup_{\mathbf{z}\in \operatorname{Opt}_c(P,Q)} \mathbb{G}_Q(\inf_{i=1, \dots, N} \{ c(\mathbf{x}_i, \mathbf{y} ) -z_i \}).$$
\item  \textbf{Two sample case $P_n,Q_m$ }
Suppose \eqref{cuadratic} and that $\frac{m}{n+m}\rightarrow \lambda\in (0,1)$,  then
$$\sqrt{\frac{nm}{n+m}}\left(\mathcal{T}_c(P_n,Q_N)- \mathcal{T}_c(P,Q)\right)\stackrel{w}\longrightarrow \sup_{\mathbf{z}\in \operatorname{Opt}_c(P,Q)} \left(\sqrt{\lambda}\mathbb{G}(\mathbf{z})+\sqrt{\lambda}\mathbb{G}_Q(\inf_{i=1, \dots, N} \{ c(\mathbf{x}_i, \mathbf{y} ) -z_i \})\right).$$

\end{itemize}
The fact that the curse of dimensionality seems to not affect the semi discrete case for both probabilities is quite astonishing. But it is partially hidden in the assumption that the set $\mathbb{X}$ has a fixed size. For a better understanding we provide in Theorem~ \ref{bound_exp} for the particular case of $\mathcal{W}_1$, a bound which studies  the effect of the choice of a discretization with size the one of the set $\mathbb{X}$. It highlights a natural trade-off between the discretization scheme of the distribution and the sampling of the distribution. \\
\indent Moreover, on the cases where $P,Q\in \mathcal{P}(\R^d)$ be such that $Q\ll\ell_d$ and its support is connected with Lebesgue negligible boundary, if the cost $c$ satisfies (A1)-(A3), all  previous limits can be made more explicit and the supremum in previous limits can be computed. Such results are given in Section 3 for such cases where the transport potential is unique up to additive constants.  In this case, under some assumption of regularity on the cost and on $Q$,  the limit is not a supremun anymore, but simply a centered Gaussian random variable. \\
\indent Finally the last section studies the semidiscrete O.T. in manifolds and gives, up to our knowledge, the first Central Limit Theorem for the solutions of the dual problem \eqref{dual}. We underline this result can not be generalised for continuous distributions. Indeed, if both probabilities are continuous and the space is not one dimensional, we cannot expect such type of central limit for the potentials, since,  the expected value of the estimation of the transport cost converges with rate $O(n^{-\frac{1}{d}})$ and no longer $O(n^{-\frac{1}{2}})$. When the two samples are discrete, even if such a rate is $O(n^{-\frac{1}{2}})$, the lack of uniqueness of the dual problem does not allow to prove such type of problems. In consequence, the semidiscrete is the unique case where such results, for the potentials of the O.T. problem in general dimension, can be expected.

\section{Central Limit Theorems for semidiscrete distributions}\label{two_samp}
\subsection{Semidiscrete optimal transport reframed as optimization program}
Consider general Polish spaces $\mathcal{X}, \mathcal{Y}$ and let $\mathcal{P}(\mathcal{Y})$ be the set of distributions on $\mathcal{Y}$. Consider also a generic finite set,  $\mathbb{X}=\{\mathbf{x}_1, \dots, \mathbf{x}_N \}\subset \mathcal{X}$ be such that $\mathbf{x}_i\neq \mathbf{x}_j$, for $i\neq j$. In all this work, we consider  $\mathcal{P}(\mathbb{X})$ the set of probabilities supported  in this finite  set. So   any $P\in \mathcal{P}(\mathbb{X})$ can be written as
\begin{equation}\label{represen}
   \text{$P:=\sum_{k=1}^Np_k\delta_{\mathbf{x}_k}$, where  $p_i>0$, for all $i=1, \dots,N$, and $\sum_{k=1}^Np_k=1$.
 } 
\end{equation}
In consequence $P$ is characterized by the vector $\mathbf{p}=(p_1, \dots, p_N)\in \R^N$.\\

We focus on semi-discret optimal transport cost which is defined as the optimal transport between a finite probability $P  \in \mathcal{P}(\mathbb{X})$ and any probability $Q \in \mathcal{P}(\mathcal{Y})$.\vskip .1in
 The following result shows that  the  optimal transport problem in the semi-discrete case  is equivalent to an optimization problem over a finite dimensional parameter space.
 Define the following function $g_c$, which   depends on  $P$ and $Q$  as 
\begin{align}
\begin{split}
    \label{defg}
    g_c(P,Q, \cdot):\R^{N}&\rightarrow \R\\
    \mathbf{z}&\mapsto g_c(P,Q, \mathbf{z})=  \sum_{i=1}^N z_ip_i+\int \min_{i=1, \dots, N} \{ c(\mathbf{x}_i, \mathbf{y} ) -z_i \}dQ(\mathbf{y}).
\end{split}
\end{align}

\begin{Lemma}\label{Lemma:dualsemi}
Let $P\in \mathcal{P}(\mathbb{X})$ , $Q\in \mathcal{P}(\mathcal{Y})$ and $c$ be a non-negative cost, then the optimal transport between $P$ and $Q$ for the cost $c$, $\mathcal{T}_c(P,Q)$, satisfies
\begin{align}\label{discrete_dual}
\mathcal{T}_c(P,Q)=\sup_{\mathbf{z}\in \R^N, \:|\mathbf{z}|\leq K^*} g_c(P,Q, \mathbf{z}). 
\end{align}
for $K^*= \frac{1}{\inf_i p_i}\left( \sup_{i=1, \dots, N}\int  c(\mathbf{y},\mathbf{x}_i) dQ(\mathbf{y}) \right)$. Moreover we can assume that $z_1=0$.
\end{Lemma}
\begin{Remark}\label{discrete_dual_remark}
Consider the dual expression of $\mathcal{T}_c(P,Q)$ and let $\varphi$ denote an optimal transport potential from $P$ to $Q$ for the cost $c$, then
$$\mathcal{T}_c(P,Q)=g_c(P,Q,\varphi(\mathbf{x}_1),\dots, \varphi(\mathbf{x}_N))).$$
Hence the optimal transport potentials and optimal values of \eqref{discrete_dual} are linked through the expression $\mathbf{z}=(\varphi(\mathbf{x}_1),\dots, \varphi(\mathbf{x}_N))$.
\end{Remark}

Note that  $g_c(P,Q, \cdot)$ is a continuous function, which can be deduced from the following lemma. And, therefore, the supremun in \eqref{discrete_dual} is attained and the
 the class of optimal values
\begin{equation}\label{opt1}
    \operatorname{Opt}_c(P,Q):=\left\lbrace  \mathbf{z}\in \R^d: \ \mathcal{T}_c(P,Q)= g_c(P,Q, \mathbf{z})\right\rbrace
\end{equation}
and its restriction 
\begin{equation}\label{opt}
    \operatorname{Opt}_c^0(P,Q):=\left\lbrace  \mathbf{z}\in \R^d: \ \mathcal{T}_c(P,Q)= g_c(P,Q, \mathbf{z}), \ \ z_1=0\right\rbrace.
\end{equation}
are both non-empty.
\begin{Lemma}\label{lemma8lipscit}
If $ f(\mathbf{y})=\inf_{i=1, \dots, N} \{ c(\mathbf{x}_i, \mathbf{y} ) -z_i \}$ and $ g(\mathbf{y})=\inf_{i=1, \dots, N} \{ c(\mathbf{x}_i, \mathbf{y} ) -s_i \}$, then
\begin{align}\label{boundLips}
    |f(\mathbf{y})-g(\mathbf{y})|\leq \sup_{i=1, \dots, N} \{|z_i-s_i| \}\leq | \mathbf{z}-\mathbf{s}|.
\end{align} 
\end{Lemma}

\subsection{Main results : Central Limit Theorems for semi-discrete optimal transport cost}
Our aim is to study the empirical semi-discrete optimal transport cost. Let $X_1, \dots, X_n$ and $Y_1, \dots, Y_m$ be two independent sequences of i.i.d. random variables with laws $P$ and $Q$ respectively, since $X_k\in \mathbb{X}$ for all $k=1,\dots,n$, the empirical measure $P_n:=\frac{1}{n}\sum_{k=1}^n\delta_{X_k}$ belongs also to $\mathcal{P}(\mathbb{X})$. In consequence it can be written as $P_n:=\sum_{k=1}^Np_k^n\delta_{\mathbf{x}_k}$, where $p_1^n, \dots, p_N^n$ are real random variables such that $p_i^n\geq 0$, for all $i=1, \dots,N$, and $\sum_{k=1}^np_k^n=1$.
 We want to study the weak limit of the following sequences corresponding to all possible  asymptotics 
$$
	\left\lbrace\sqrt{n}\left(\mathcal{T}_c(P_n,Q)- \mathcal{T}_c(P,Q)\right)\right\rbrace_{n\in \N}, \ \ 	\left\lbrace\sqrt{n}\left(\mathcal{T}_c(P,Q_m)- \mathcal{T}_c(P,Q)\right)\right\rbrace_{m\in \N},$$ 
and the	two sample case
	$$
	\left\lbrace\sqrt{\frac{nm}{n+m}}\left(\mathcal{T}_c(P_n,Q_m)- \mathcal{T}_c(P,Q)\right)\right\rbrace_{m,n\in \N},
$$
under the assumption $\frac{m}{n+m}\rightarrow \lambda \in (0,1)$.   \\

To state the asymptotic behaviour we introduce  first  a centered Gaussian vector,  $(X_1,\dots, X_N)$ with covariance matrix
\begin{equation}\label{sigma_Natrix}
   \Sigma(\mathbf{p}):= \begin{bmatrix}
p_1(1-p_1) & -p_1p_2  & \cdots& -p_1p_N\\
-p_2p_1 & p_2(1-p_2) & \cdots&-p_2p_N\\
\vdots & \vdots & \ddots&\vdots\\
-p_Np_1 &  \cdots&p_Np_{N-1}&p_N(1-p_N)
\end{bmatrix}.
\end{equation}
We also define   a centered Gaussian processes $\mathbb{G}^c_Q$  in $\R^N$ with covariance function
\begin{align}
    \begin{split}
        \label{xi_Natrix}
\Xi^c_Q(\mathbf{z},\mathbf{s}):=& \int \inf_{i=1, \dots, N} \{ c(\mathbf{x}_i, \mathbf{y} )-z_i\}\inf_{i=1, \dots, N} \{ c(\mathbf{x}_i, \mathbf{y} )-s_i\} dQ(\mathbf{y})\\
&-\int \inf_{i=1, \dots, N} \{ c(\mathbf{x}_i, \mathbf{y} )-z_i\}dQ(\mathbf{y})\int\inf_{i=1, \dots, N} \{ c(\mathbf{x}_i, \mathbf{y} )-s_i\} dQ(\mathbf{y}).
    \end{split}
\end{align}
We can now state our main theorem.
\begin{Theorem}\label{Teoremaprinci}
Let $P\in \mathcal{P}(\mathbb{X})$, $Q\in \mathcal{P}(\mathcal{Y})$, $c$ be non-negative and
\begin{align}\label{cond}
\int c(\mathbf{x}_i, \mathbf{y})dQ(\mathbf{y})<\infty, \ \text{ for all $i=1,\dots, m$},
\end{align} 
then the following limits hold.
\begin{itemize}
    \item \textbf{(One sample case for empirical discrete distribution $P_n$)}
    \begin{align*}
	\sqrt{n}\left(\mathcal{T}_c(P_n,Q)- \mathcal{T}_c(P,Q)\right)\xrightarrow{w} \sup_{\mathbf{z}\in \operatorname{Opt}_c^0(P,Q)} \sum_{i=1}^N z_i  X_i.
\end{align*}
\end{itemize}
Suppose that
\begin{align}\label{cuadratic}
\int c(\mathbf{x}_i, \mathbf{y})^2dQ(\mathbf{y})<\infty, \ \text{ for all $i=1,\dots,N$}.
\end{align} 
\begin{itemize}
 \item \textbf{(0ne sample case for empirical distribution $Q_m$)}
$$\sqrt{m}\left(\mathcal{T}_c(P,Q_m)- \mathcal{T}_c(P,Q)\right)\stackrel{w}\longrightarrow \sup_{\mathbf{z}\in \operatorname{Opt}_c^0(P,Q)} \ \mathbb{G}^c_Q(\mathbf{z}).$$
\item  \textbf{(Two sample case )}
if $n,m\rightarrow\infty$, with $\frac{m}{n+m}\rightarrow \lambda\in  (0,1)$,  then
$$\sqrt{\frac{nm}{n+m}}\left(\mathcal{T}_c(P_n,Q_m)- \mathcal{T}_c(P,Q)\right)\stackrel{w}\longrightarrow \sup_{\mathbf{z}\in \operatorname{Opt}_c^0(P,Q)} \left(\sqrt{\lambda} \sum_{i=1}^N z_i  X_i+(\sqrt{1-\lambda})\mathbb{G}^c_Q(\mathbf{z})\right).$$
\end{itemize}
Here $(X_1,\dots,X_N)\sim \mathcal{N}(\mathbf{0}),$ for $\Sigma(\mathbf{p}) $ defined in \eqref{sigma_Natrix},  and  $\mathbb{G}^c_Q$ is a centered Gaussian process with covariance function $\Xi^c(Q)$ defined in \eqref{xi_Natrix}. Moreover $\mathbb{G}^c_Q$ and $(X_1,\dots,X_N)$ are independent.

\end{Theorem}

When $\mathcal{X}$ and $\mathcal{Y}$ are contained in the same Polish space $(\mathcal{Z},d)$, a particular cost that satisfies the assumptions of Theorem \ref{Teoremaprinci} is the metric $d^p$ for all $p\geq 1$. Then applying Theorem~\ref{Teoremaprinci} to the empirical estimations of ${\mathcal{T}_{d^p}(P,Q)}$  and a delta-method, enable to prove the asymptotic behaviour of the $p$-Wasserstsein distance  $\mathcal{W}_p^p(P,Q)={\mathcal{T}_{d^p}(P,Q)}$ as given in the following corollary.
\begin{Corollary}\label{coro:semi_disc}
Let $P\in \mathcal{P}(\mathbb{X})$ and $Q\in \mathcal{P}(\mathcal{Z}) $ be such that 
\begin{align}\label{assumptio:1_disc}
\int d( \mathbf{x}_0, \mathbf{y})^pdQ(\mathbf{y})<\infty, \ \text{ for some $\mathbf{x}_0\in \mathbb{X}$}.
\end{align}
 Then, for any $p\geq 1$, we have
\begin{itemize}
    \item \textbf{(One sample case for $P$)}
 \begin{itemize}
     \item if $\mathcal{W}_p(P,Q)\neq 0$, 
     \begin{align*}
	\sqrt{n}\left(\mathcal{W}_p(P_n,Q)- \mathcal{W}_p(P,Q)\right)\overset{w}\longrightarrow \frac{1}{p\left(\mathcal{W}_p(P,Q)\right)^{p-1}}\sup_{\mathbf{z}\in \operatorname{Opt}_{d^p}(P,Q)} \sum_{i=1}^N z_i  X_i,
\end{align*}

     \item if $\mathcal{W}_p(P,Q)= 0$, 
          \begin{align*}
	n^{\frac{1}{2p}}\mathcal{W}_p(P_n,Q)\overset{w}\longrightarrow\left(\sup_{\mathbf{z}\in \operatorname{Opt}_{d^p}(P,P)} \sum_{i=1}^N z_i  X_i\right)^{\frac{1}{p}}.
\end{align*}
 \end{itemize}
\end{itemize}
Suppose that
\begin{align}\label{assumptio:1_disc2}
\int d( \mathbf{x}_0, \mathbf{y})^{2p}dQ(\mathbf{y})<\infty, \ \text{ for some $\mathbf{x}_0\in \mathbb{X}$},
\end{align}
then
\begin{itemize}
 \item \textbf{(0ne sample case for $Q$)}
  \begin{itemize}
     \item if $\mathcal{W}_p(P,Q)\neq 0$, 
     \begin{align*}
	\sqrt{m}\left(\mathcal{W}_p(P,Q_m)- \mathcal{W}_p(P,Q)\right)\overset{w}\longrightarrow \frac{1}{p\left(\mathcal{W}_p(P,Q)\right)^{p-1}} \sup_{\mathbf{z}\in \operatorname{Opt}_c^0(P,Q)}\mathbb{G}^{d^p}_Q(\mathbf{z}),
\end{align*}
     \item if $\mathcal{W}_p(P,Q)= 0$, 
          \begin{align*}
	m^{\frac{1}{2p}}\mathcal{W}_p(P,Q_m)\overset{w}\longrightarrow\left( \sup_{\mathbf{z}\in \operatorname{Opt}_c^0(P,Q)} \mathbb{G}^{d^p}_Q(\mathbf{z})\right)^{\frac{1}{p}}.
\end{align*}
 \end{itemize}
\item  \textbf{(Two sample case)}
if $n,m\rightarrow\infty$ with $\frac{m}{n+m}\rightarrow \lambda\in  (0,1)$,  then
\begin{itemize}
     \item if $\mathcal{W}_p(P,Q)\neq 0$, 
     \begin{align*}
	\sqrt{\frac{nm}{n+m}}\left(\mathcal{W}_p(P_n,Q_m)- \mathcal{W}_p(P,Q)\right)\overset{w}\longrightarrow \frac{\sup_{\mathbf{z}\in \operatorname{Opt}_c^0(P,Q)} \left(\sqrt{\lambda} \sum_{i=1}^N z_i  X_i+(\sqrt{1-\lambda})\mathbb{G}^{d^p}_Q(\mathbf{z})\right)}{p\left(\mathcal{W}_p(P,Q)\right)^{p-1}} ,
\end{align*}
     \item if $\mathcal{W}_p(P,Q)= 0$, 
          \begin{align*}
	{\frac{nm}{n+m}}^{\frac{1}{2p}}\mathcal{W}_p(P,Q_m)\overset{w}\longrightarrow\left( \sup_{\mathbf{z}\in \operatorname{Opt}_c^0(P,Q)} \left(\sqrt{\lambda} \sum_{i=1}^N z_i  X_i+(\sqrt{1-\lambda})\mathbb{G}^{d^p}_Q(\mathbf{z})\right)\right)^{\frac{1}{p}},
\end{align*}
 \end{itemize}
\end{itemize}
Here $(X_1,\dots,X_N)\sim \mathcal{N}(\mathbf{0},\Sigma(\mathbf{p}) $, for $\Sigma(\mathbf{p}) $ defined in \eqref{sigma_Natrix},  and  $\mathbb{G}^{d^p}_Q$ is a centered Gaussian process with covariance function $\Xi^{d^p}(Q)$, defined in \eqref{xi_Natrix}.
Moreover, $\mathbb{G}^{d^p}_Q$ and $(X_1,\dots,X_N)$ are independent.
\end{Corollary}

The next subsection proves Theorem~\ref{Teoremaprinci} in the two sample case. The same proof verbatim applies also for the CLT for the one sample case for $Q$. The  one sample case for $P$ can be proven under weaker moment assumptions on $Q$ and will be commented separately.  
\subsection{Proof of Theorem~\ref{Teoremaprinci}}
The strategy of the proof is the following, first we start by proving the central limit theorem for bounded potentials. That means the study of the  asymptotic behaviour of the sequence 
$$\sqrt{\frac{nm}{n+m}} \left(\sup_{|\mathbf{z}|\leq K} g_c(P_n,Q_m, \mathbf{z})-\sup_{|\mathbf{z}|\leq K} g_c(P,Q, \mathbf{z})\right)_{n,m}.$$
The weak limit depends on the set of restricted optimal points.
$$ \operatorname{Opt}_c^K(P,Q):=\left\lbrace\mathbf{z}: \ \sup_{|\mathbf{s}|\leq K} g_c(P,Q, \mathbf{s})=g_c(P,Q, \mathbf{z}), \ \ z_0=0\right\rbrace.$$
\begin{Lemma}\label{Lemma_two_samples}
Set $K>0$, under the assumptions of Theorem~\ref{Teoremaprinci}, we have the limit
$$\sqrt{\frac{nm}{n+m}} \left(\sup_{|\mathbf{z}|\leq K} g_c(P_n,Q_m, \mathbf{z})-\sup_{|\mathbf{z}|\leq K} g_c(P,Q, \mathbf{z})\right)\stackrel{w}\longrightarrow \sup_{\mathbf{z}\in \operatorname{Opt}_c^K(P,Q)}  \left(\sqrt{\lambda} \sum_{i=1}^N z_i  X_i+(\sqrt{1-\lambda})\mathbb{G}^c_Q(\mathbf{z})\right),$$
with  $(X_1,\dots,X_N)$ and $\mathbb{G}^c_Q$ as in Theorem~\ref{Teoremaprinci}.
\end{Lemma}
\begin{proof}[Proof of Lemma~\ref{Lemma_two_samples}.]
For each $K>0$ we define the restricted set  
$$ \mathcal{F}_c^K=\left\lbrace \mathbf{y}\mapsto \inf_{i=1, \dots, N} \{ c(\mathbf{x}_i, \mathbf{y} ) -z_i \} , \ \ \mathbf{x}_i \in \mathbb{X}, \ \ \text{and} \ \ z_1 =0, \ | \mathbf{z}|\leq K\right\rbrace,$$
 Lemma \ref{covering} proves that such a class is $Q$-Donsker, see  Theorem 1.5.7 in \cite{Vart_Well}, in the sense that
$$\sqrt{m}(Q_m-Q) \stackrel{w}\longrightarrow \mathbb{G}_Q\ \ \text{in} \ \ \ell^{\infty}(\mathcal{F}_c^K),$$ 
where $\mathbb{G}_Q$ is the Brownian bridge in $\mathcal{F}_c^K$. This is a centered Gaussian process with covariance function
$$(f,g)\mapsto \int f(\mathbf{y})g(\mathbf{y})dQ(\mathbf{y})-\int f(\mathbf{y})dQ(\mathbf{y})\int g(\mathbf{y})dQ(\mathbf{y}).$$ 
Let $\bar{\mathbb{B}}_K(\mathbf{0})$ be the closure of the centered ball of radius $K$ in $\R^N$. Note that the functional 
\begin{align*}
   C: \ell^{\infty}(\mathcal{F}_c^K)&\longrightarrow  \ell^{\infty}(\bar{\mathbb{B}}_K(\mathbf{0}))\\
   f&\mapsto \left(\mathbf{z}\mapsto f\left( \inf_{i=1, \dots, N} \{ c(\mathbf{x}_i, \mathbf{y} ) -z_i\}\right)\right)
\end{align*}
is actually continuous,  hence for any $f, g\in \ell^{\infty}(\mathcal{F}_c^K)$, we have
\begin{align*}
   \sup_{\mathbf{z}\in \mathbb{B}_K(\mathbf{0})}\left|f\left( \inf_{i=1, \dots, N} \{ c(\mathbf{x}_i, \mathbf{y} ) -z_i\}\right)- g\left( \inf_{i=1, \dots, N} \{ c(\mathbf{x}_i, \mathbf{y} ) -z_i\}\right)\right|=\sup_{\phi \in \ell^{\infty}(\mathcal{F}_c^K)}|f(\phi)-g(\phi)|.
\end{align*}
Moreover, the multivariate CLT implies 
$
    \sqrt{n}\left( \mathbf{p}_n- \mathbf{p}\right)\stackrel{w}\longrightarrow  (X_1,\dots,X_N)\sim N\left(\mathbf{0}, \Sigma(\mathbf{p})\right),
$
where $\Sigma(\mathbf{p})$ is defined in \eqref{sigma_Natrix}. Since the sequences $\sqrt{n}(\mathbf{p}_n-\mathbf{p})$ and $\sqrt{m}(Q_m-Q)$ are  independent we derive the following result.
\begin{Lemma}\label{lemma:Weak_limit_nifotm}
Under the assumptions of Theorem~\ref{Teoremaprinci}, we have the limit 
\begin{equation}\label{rewrittinCLT2}
     \sqrt{\frac{nm}{n+m}} \left( g_c(P_n,Q_m, \cdot)-g_c(P,Q, \cdot)\right)\stackrel{w}\longrightarrow \sqrt{\lambda}\langle{\mathbf{X}}, \cdot\rangle+\sqrt{1-\lambda}C(\mathbb{G}_Q)) \ \ \text{in }\ \ell^{\infty}(\bar{\mathbb{B}}_K(\mathbf{0})),
\end{equation}
with $(X_1,\dots,X_N)=\mathbf{X}$.
\end{Lemma}

Let $(\mathcal{B}, d)$ be a compact metric space,  Corollary~2.3 in \cite{Luis2020}, provides the  directional Hadamard derivative of the functional 
\begin{align*}
    \delta: \ell^{\infty}(\mathcal{B}) &\longrightarrow \R\\
    F&\mapsto \delta(F)=\sup_{\mathbf{z}\in \mathcal{B} }F(\mathbf{z}),
\end{align*}
tangentially to $\mathcal{C}(\mathcal{B})$ (the space of continuous functions from $\mathcal{B}$ to $\R$ ) with respect to $F$ in a direction $G\in \mathcal{C}(\mathcal{B})$.  Recall that a function $f:\Theta\rightarrow \R$, defined in a Banach space, $\Theta$, is said to be \emph{Hadamard directionally differentiable} at $\theta\in \Theta$ tangentially to $\Theta_0\subset \Theta$  if there exists a a function $f'_{\theta}:\Theta_0\rightarrow \R$ such that
\begin{equation*}
    \frac{f(\theta+t_n{h}_n)-f(\theta)}{t_n}\xrightarrow[n\rightarrow\infty]{} f'_{\theta}({h}),\ \ \text{for all sequences $t_n\searrow 0$ and ${h}_n\rightarrow {h}$, for all $h\in \Theta_0$.}
\end{equation*}
If $F\in \mathcal{C}(\mathcal{B}) $ is not identically $0$, the precise formula for the derivative, provided by Corollary~2.3 in \cite{Luis2020}, is
\begin{equation}\label{LuisDeriv}
    \delta'_{F}(G)=\sup_{\{\mathbf{z}: \ F(\mathbf{z})=\delta(F)\}}G(\mathbf{z}), \ \ \text{for  $G\in \mathcal{C}(\mathcal{B})$}.
\end{equation}
In our case the compact metric space is the ball $\bar{\mathbb{B}}_K(\mathbf{0})$, the functional $F$ correspond with $g_c(P,Q, \cdot)$ and the set of optimal points is $ \operatorname{Opt}_c^K(P,Q)$. The following result rewrites  \eqref{LuisDeriv} in our setting. 
\begin{Lemma}
\label{LuisDerivLemma}
Set $K>0$, under the assumptions of Theorem~\ref{Teoremaprinci}, the map $\delta$ is Hadamard directionally differentiable at $ g_c(P,Q, \cdot) $, tangentially to the set 
$\mathcal{C}(\bar{\mathbb{B}}_K(\mathbf{0}))$ with derivative, for $G\in \mathcal{C}(\bar{\mathbb{B}}_K(\mathbf{0}))$,
\begin{equation*}
    \delta'_{g_c(P,Q, \cdot)}(G)=\sup_{\mathbf{z}\in \operatorname{Opt}_c^K(P,Q)}G(\mathbf{z}).
\end{equation*}
\end{Lemma}
The last step is the application of the delta-method. Let $\Theta$ be a Banach space,  $\theta\in \Theta$ and $\{Z_n\}_{n\in \N}$ be a sequence of random variables such that $Z_n:\Omega_n\rightarrow \Theta $ and $r_n(Z_n-\theta)\stackrel{w}{\longrightarrow} Z$ for some sequence $r_n\rightarrow +\infty$ and some random element $Z$ that takes values in $\Theta_0\subset\Theta$. If  $f:\Theta\rightarrow\R$ is Hadamard differentiable at $\theta $ tangentially to $\Theta_0\subset \Theta$,  with derivative $f'_{\theta}(\cdot):\Omega_0\rightarrow \R$, then Theorem~1 in \cite{Romisch}, so-called delta-method, states that  $r_n(f(Z_n)-f(\theta))\stackrel{w}{\longrightarrow} f'_{\theta}(Z)$.
\\
%\color{green}

Now, it only remains to prove that the limit in \eqref{rewrittinCLT} belongs to $\mathcal{C}(\bar{\mathbb{B}}_K(\mathbf{0}))$.  Such a limit is a mixture of two independent processes. The first one $\langle{\mathbf{X}}, \cdot\rangle$ has clearly continuous sample paths with respect to the euclidean norm $| \cdot|$ in $\R^N$.  On the other side, $\mathbb{G}_Q$ has continuous sample paths in $\mathcal{F}_c^K$ with respect to the semi-metric 
$$\rho_Q(f)=  \int f(\mathbf{y})^2dQ(\mathbf{y})-\left(\int f(\mathbf{y})dQ(\mathbf{y})\right)^2,$$
in the sense that, see pag 89  in \cite{Vart_Well}, there exists some sequence $\delta_n\searrow 0$ such that 
\begin{equation}
    \label{pcountinous}
     \sup_{f,g\in\mathcal{F}_c^K, \ \rho_Q(f,g)<\delta_n }|\mathbb{G}_Q(f)-\mathbb{G}_Q(g)|\stackrel{a.s.}\longrightarrow 0.
\end{equation}

We want now to analyse the value 
$
\sup_{|\mathbf{z}-\mathbf{s}|<\delta_n}|C(\mathbb{G}_Q)(\mathbf{z})-C(\mathbb{G}_Q)(\mathbf{s})|
$
Note that for every $f\in \mathcal{F}_c^K  $ there exists some $\mathbf{z}^{f}\in \bar{\mathbb{B}}_K(\mathbf{0})$ such that $f(\mathbf{y})=\inf_{i=1, \dots, m} \{ c(\mathbf{x}_i, \mathbf{y} ) -z_i^f \}$. Lemma~\ref{lemma8lipscit} states that 
\begin{equation}
    \label{subset}
    \{f,g\in \mathcal{F}_c^K: |\mathbf{z}^{f}-\mathbf{z}^{g}|<\delta\}\subset \{f,g\in \mathcal{F}_c^K: ||f-g||_{\infty}<\delta\}\subset\{f,g\in\mathcal{F}_c^K: \rho_Q(f,g)<\delta\}.
\end{equation}
Since 
$
    |C(\mathbb{G}_Q)(\mathbf{z}^f)-C(\mathbb{G}_Q)(\mathbf{z}^g)|=|\mathbb{G}_Q(f)-\mathbb{G}_Q(g)|,
$
then we have 
\begin{align*}
   \sup_{|\mathbf{z}^f-\mathbf{z}^g|<\delta_n} |C(\mathbb{G}_Q)(\mathbf{z}^f)-C(\mathbb{G}_Q)(\mathbf{z}^g)|=\sup_{|\mathbf{z}^f-\mathbf{z}^g|<\delta_n}|\mathbb{G}_Q(f)-\mathbb{G}_Q(g)|,
\end{align*}
and, consequently, using \eqref{subset} and \eqref{pcountinous}, we obtain
\begin{align*}
   \sup_{|\mathbf{z}^f-\mathbf{z}^g|<\delta_n} |C(\mathbb{G}_Q)(\mathbf{z}^f)-C(\mathbb{G}_Q)(\mathbf{z}^g)|\leq \sup_{\rho_Q(f,g)<\delta_n}|\mathbb{G}_Q(f)-\mathbb{G}_Q(g)|\stackrel{a.s.}\longrightarrow 0.
\end{align*}

Finally, Lemma \ref{lemma:Weak_limit_nifotm} implies that
$
     \sqrt{\frac{nm}{n+m}} \left( g_c(P_n,Q_m, \cdot)-g_c(P,Q, \cdot)\right)
$
has a weak limit $Z$ in $\ell^{\infty}(\bar{\mathbb{B}}_K(\mathbf{0}))$ having a version in $\mathcal{C}(\bar{\mathbb{B}}_K(\mathbf{0}))$.
Applying the so-called delta-method to the function $\delta$ and Lemma~\ref{LuisDerivLemma} we derive the limit
$$\sqrt{\frac{nm}{n+m}} \left(\sup_{|\mathbf{z}|\leq K} g_c(P_n,Q_m, \mathbf{z})-\sup_{|\mathbf{z}|\leq K} g_c(P,Q, \mathbf{z})\right)\stackrel{w}\longrightarrow \sup_{\mathbf{z}\in \operatorname{Opt}_c^K(P,Q)} Z(\mathbf{z}).$$
Note, that the process 
$\mathbf{z}\mapsto 
\sqrt{1-\lambda}\mathbb{G}_Q\left( \inf_{i=1, \dots, N} \{ c(\mathbf{x}_i, \mathbf{y} ) -z_i\}\right)$ is Gaussian in $\R^N$ with covariance function $\Xi^c_Q$. Moreover, it is independent from $\mathbf{X}$, then the law of the process $Z$ is the same of  the process $\sqrt{\lambda} \langle\mathbf{X}, \cdot\rangle+(\sqrt{1-\lambda})\mathbb{G}^c_Q$ and the theorem holds.
\end{proof}
$ $\\

Unfortunately, the optimal solutions need not be universally bounded. In order to go from the bounded to the unbounded, we observe that Lemma~\ref{Lemma:dualsemi} implies
\begin{align*}
\mathcal{T}_c(P_n,Q_m)=\sup_{\mathbf{z}\in \R^N} g_c(P_n,Q_m, \mathbf{z})=\sup_{|\mathbf{z}|\leq K_{n,m}} g_c(P_n,Q_m, \mathbf{z}),
\end{align*}
for $K_{n,m}= \frac{1}{\inf_i p_i^n}\left( \sup_{i=1, \dots, N}\int  c(\mathbf{y},\mathbf{x}_i) dQ_m(\mathbf{y}) \right)$. Let $K^*$ be the constant provided in Lemma~\ref{Lemma:dualsemi} for $P$ and $Q$ (that means $K^*= \frac{1}{\inf_i p_i}\left( \sup_{i=1, \dots, N}\int  c(\mathbf{y},\mathbf{x}_i) dQ(\mathbf{y}) \right)$). The strong law of large numbers implies the a.s. convergence of $\inf_i p_i^n$ to $\inf_i p_i$, and, assuming \eqref{cuadratic}, we have that the sequence
$
   \sqrt{m}\left( K_{n,m}-K^*\right)
$
is stochastically bounded. Finally,  the difference  $\sqrt{\frac{nm}{n+m}}\left(\mathcal{T}_c(P_n,Q_m)- \mathcal{T}_c(P,Q)\right)$~is equal  to
\begin{equation}\label{rewritten}
    \sqrt{\frac{nm}{n+m}}\left(\mathcal{T}_c(P_n,Q_m)-\sup_{|\mathbf{z}|\leq K^*+1} g_c(P_n,Q_m, \mathbf{z})\right) +  \sqrt{\frac{nm}{n+m}}\left(\sup_{|\mathbf{z}|\leq K^*+1} g_c(P_n,Q_m, \mathbf{z})-\mathcal{T}_c(P,Q)\right),
\end{equation}
and Lemma~\ref{Lemma_two_samples} implies the weak convergence of the second term to
$$\sup_{\mathbf{z}\in \operatorname{Opt}_c^{K^*+1}(P,Q)} \left(\sqrt{\lambda} \sum_{i=1}^N z_i  X_i+(\sqrt{1-\lambda})\mathbb{G}^c_Q(\mathbf{z})\right)=\sup_{\mathbf{z}\in \operatorname{Opt}_c^0(P,Q)} \left(\sqrt{\lambda} \sum_{i=1}^N z_i  X_i+(\sqrt{1-\lambda})\mathbb{G}^c_Q(\mathbf{z})\right), $$
where the equality is a direct consequence of Lemma~\ref{Lemma:dualsemi}. It only remains to prove that the first term of \eqref{rewritten} tends to $ 0$ in probability. Note that we have two cases.
\begin{itemize}
    \item The first one is $K_{n,m}\leq K^*+1$, which implies that 
    $$ \mathcal{T}_c(P_n,Q_m)=\sup_{|\mathbf{z}|\leq K_{n,m}} g_c(P_n,Q_m, \mathbf{z}) =\sup_{|\mathbf{z}|\leq K^*+1} g_c(P_n,Q_m, \mathbf{z})\leq \mathcal{T}_c(P_n,Q_m),$$
    and makes $0$ the first term of \eqref{rewritten}.
    \item The second one is $K_{n,m}\geq K^*+1$, which implies the bound
    \begin{equation}\label{equqtionbounditem}
        0\leq \mathcal{T}_c(P_n,Q_m)-\sup_{|\mathbf{z}|\leq K^*+1} g_c(P_n,Q_m, \mathbf{z})\leq 
    \sup_{|\mathbf{z}|\leq K_{n,m}} g_c(P_n,Q_m, \mathbf{z})-\sup_{|\mathbf{z}|\leq K^*+1} g_c(P_n,Q_m, \mathbf{z}).
    \end{equation}
    Note that the right side of the inequality \eqref{equqtionbounditem}  can be rewritten as 
    $$  \sup_{|\mathbf{z}|\leq K_{n,m}} g_c(P_n,Q_m, \mathbf{z})-\sup_{|\mathbf{z}|\leq K^*+1} g_c(P_n,Q_m, \mathbf{z})=\sup_{|\mathbf{z}|\leq K_{n,m}} \inf_{|\mathbf{z}'|\leq K^*+1} g_c(P_n,Q_m, \mathbf{z})- g_c(P_n,Q_m, \mathbf{z}')$$
    and upper bounded by 
    $$ \sup_{|\mathbf{z}|\leq K_{n,m}} \inf_{|\mathbf{z}'|\leq K^*+1} |g_c(P_n,Q_m, \mathbf{z})- g_c(P_n,Q_m, \mathbf{z}')|.$$
    Since
    \begin{align*}
    |g_c(P_n,Q_m, \mathbf{z})&- g_c(P_n,Q_m, \mathbf{z}')|\\
    &\leq
     \sum_{i=1}^N |z_i-z_i'|p_i^n+\int |\inf_{i=1, \dots, N} \{ c(\mathbf{x}_i, \mathbf{y} ) -z_i \}-\inf_{i=1, \dots, N} \{ c(\mathbf{x}_i, \mathbf{y} ) -z_i' \}|dQ_m(\mathbf{y}),
     \end{align*}
     we can conclude from \eqref{boundLips} that
     $$ 0\leq \mathcal{T}_c(P_n,Q_m)-\sup_{|\mathbf{z}|\leq K^*+1} g_c(P_n,Q_m, \mathbf{z})\leq 2\sup_{|\mathbf{z}|\leq K_{n,m}} \inf_{|\mathbf{z}|\leq K^*+1} |\mathbf{z}-\mathbf{z}'|\leq |K_{n,m} - K^*-1|.$$
\end{itemize}
Both cases together yield the inequality 
\begin{equation}
    \label{inequqltybeforewrong}
    0 \leq \mathcal{T}_c(P_n,Q_m)-\sup_{|\mathbf{z}|\leq K^*+1} g_c(P_n,Q_m, \mathbf{z})\leq |K_{n,m} - K^*-1| \mathbbm{1}_{(K_{n,m}\geq K^*+1)}.
\end{equation}
To see that the $\sqrt{\frac{nm}{n+m}}|K_{n,m} - K^*-1| \mathbbm{1}_{(K_{n,m}\geq K^*+1)}$ tends to $0$ in probability, we write 
$|K_{n,m} - K^*-1| \mathbbm{1}_{(K_{n,m}\geq K^*+1)}=\max(0,K_{n,m} - K^*-1 ).$
Note that 
$$ \sqrt{\frac{nm}{n+m}}\max(0,K_{n,m} - K^*-1 )=\max(0,\sqrt{\frac{nm}{n+m}}\left(K_{n,m} - K^*-1 \right)).$$
Since $\frac{nm}{n+m}\left(K_{n,m} - K^*\right)$ is stochastically  bounded implies that $ \sqrt{\frac{nm}{n+m}}\left(K_{n,m} - K^*-1 \right) $ converges to $-\infty$ in probability and 
$$\sqrt{\frac{nm}{n+m}}|K_{n,m} - K^*-1| \mathbbm{1}_{(K_{n,m}\geq K^*+1)} =\max(0,\sqrt{\frac{nm}{n+m}}\left(K_{n,m} - K^*-1 \right))\stackrel{P}\longrightarrow 0 .$$
That proves Theorem \ref{Teoremaprinci}.

%%% J en suis la 

\begin{Remark}
When dealing with the case where the asymptotics depend only on the empirical distribution  $P_n$, note that  Assumption~\eqref{cond}, which depends only on $Q$ 
\begin{align*}
\int c(\mathbf{x}_i, \mathbf{y})dQ(\mathbf{y})<\infty, \ \text{ for all $i=1,\dots, m$},
\end{align*} 
is enough to prove the CLT. Actually,  the multidimensional CLT yields that  
\begin{equation}\label{rewrittinCLT}
     \sqrt{n} \left( g_c(P_n,Q, \cdot)-g_c(P,Q, \cdot)\right)\stackrel{w}\longrightarrow \langle{\mathbf{X}}, \cdot\rangle \ \ \text{in }\ \ell^{\infty}(\bar{\mathbb{B}}_K(\mathbf{0})),
\end{equation}
with $(X_1,\dots,X_N)=\mathbf{X}$. Therefore, all of the previous reasonings can be now repeated verbatim.
\end{Remark}
\subsection{An upper-bound on the expectation}
Theorem \ref{Teoremaprinci} states the central limit theorem, when one of both probabilities is supported on a finite set. Now, we investigate the influence of the number of points of the discrete measure on the convergence bounds. In order to better understand the influence of the number of points, we will restrict our analysis to the euclidean cost. 

\begin{Theorem}\label{bound_exp}
Let $P$ be  supported on $N$ points in $\mathbb{X}$,  $Q\in \mathcal{P}(\mathcal{Y})$ be a distribution  with finite second order moment and $Q_m$ its corresponding empirical version, then
\begin{equation*}
    E\left|\mathcal{W}_1(P,Q_m)-\mathcal{W}_1(P,Q)\right|\leq \frac{8\sqrt{2N}}{\sqrt{m}}K(\mathbb{X}, Q)
\end{equation*}
where 
$$K(\mathbb{X}, Q)=(4\operatorname{diam}(\mathbb{X})+2 \sqrt{ \int | \mathbf{y}|^2dQ(\mathbf{y})}+2\operatorname{diam}(\mathbb{X})) \left(\log(2)+\sqrt{{2 \operatorname{diam}(\mathbb{X})+ 1}}\right). $$
\end{Theorem}
\begin{proof}
Let $Y_1, \dots, Y_{m}$ be i.i.d with law $Q$. Recall that, when the cost $c$ is the euclidean distance $|\cdot|$, then the optimal transport potentials are $1$-Lipschitz functions. This yields trivially 
\begin{align}\label{discrete_dual_2}
\mathcal{W}_1(P,Q)=\sup_{\mathbf{z}\in \R^N} g_1(P,Q, \mathbf{z})=\sup_{|\mathbf{z}|\leq \operatorname{diam}(\mathbb{X})} g_1(P,Q, \mathbf{z}),
\end{align}
where
$  g_1(P,Q, \mathbf{z})=  \sum_{i=1}^N z_ip_i+\int \inf_{i=1, \dots, N} \{ |\mathbf{x}_i-\mathbf{y} | -z_i \}dQ(\mathbf{y}).$ We want to bound the quantity
$ E\left|\mathcal{W}_1(P,Q_m)-\mathcal{W}_1(P,Q)\right|$, which can be rewritten, by \eqref{discrete_dual_2}, as 
\begin{align*}
     E\left|\sup_{|\mathbf{z}|\leq \operatorname{diam}(\mathbb{X})} g_1(P,Q, \mathbf{z})-\sup_{|\mathbf{z}|\leq \operatorname{diam}(\mathbb{X})} g_1(P,Q_m, \mathbf{z})\right|,
\end{align*}
and upper bounded by 
$
    E\left|\sup_{f\in \mathcal{F}_1}\int f(\mathbf{y}) (dQ_m(\mathbf{y}) -dQ(\mathbf{y}) )\right|,
$
where
$$ \mathcal{F}_1=\left\lbrace \mathbf{y}\mapsto \inf_{i=1, \dots, N} \{ c(\mathbf{x}_i, \mathbf{y} ) -z_i \} , \ \ \mathbf{x}_i \in \mathbb{X}, \ \ \text{and} \ \ | \mathbf{z}|\leq \operatorname{diam}(\mathbb{X})\right\rbrace.$$
We set $D=\operatorname{diam}(\mathbb{X})$ in order to simplify the following formulas. Denote by ${N}(\epsilon, \mathcal{F}_1, ||\cdot ||_{L^2(Q_m)})$  the he covering number with respect to the  metric $L^2(Q_m).$ Lemma 4.14 in \cite{Massart2007ConcentrationIA} and Lemma~\ref{lemma8lipscit} imply that 
\begin{equation*}
    \sqrt{ \log\left(2{N}(\epsilon, \mathcal{F}_1, ||\cdot ||_{L^2(Q_m)})\right)}\leq \sqrt{ N \log\left(\frac{2{D}}{\epsilon}+1\right)+\log(2)}.
\end{equation*}
Let denote as $a_{N,m}$ the (random) quantity 
$a_{N,m}=2\sqrt{\int \sup_{i=1, \dots, N}|\mathbf{x}_i- \mathbf{y}|^2dQ_m(\mathbf{y})}, $
then
\begin{align}
    \begin{split}
      \label{covering_bound}
   \int_{0}^{{2{D}+a_{N,m}}} \sqrt{ \log\left(2{N}(\epsilon, \mathcal{F}_1, ||\cdot ||_{L^2(Q_m)})\right)}d\epsilon\leq & \sqrt{N} \int_{0}^{{{D}+a_{N,m}}}  \sqrt{  \log\left(\frac{2{D}}{\epsilon}+1\right)}d\epsilon\\&+{D} \log(2).  
    \end{split}
\end{align}

Now recall by Theorem 3.5.1 in \cite{Gin2015MathematicalFO} that 
\begin{align*}
   & \sqrt{m}E\left|\sup_{f\in \mathcal{F}_1}\int f(\mathbf{y}) (dQ_m(\mathbf{y}) -dQ(\mathbf{y}) )\right|\\
    &\leq 8\sqrt{2}E\left|  \int_{0}^{{2{D}+a_{N,m}}}  \sqrt{ \log\left(2{N}(\epsilon, \mathcal{F}_1, ||\cdot ||_{L^2(Q_m)})\right)}d\epsilon \right|\\
    &\leq 8\sqrt{2}\sqrt{N} E\int_{0}^{{2{D}+a_{N,m}}} \sqrt{  \log\left(\frac{2{D}}{\epsilon}+1\right)}d\epsilon+E({{2{D}+a_{N,m}}} ) \log(2),
\end{align*}
where the last inequality is consequence of \eqref{covering_bound}. Since, using Jensen's inequality, 
\begin{align*}
    &\int_{0}^{{2{D}+a_{N,m}}} \sqrt{ \log\left(\frac{2{D}}{\epsilon}+1\right)}d\epsilon\\&\leq ({2{D}+a_{N,m}}) \sqrt{\frac{2 {D} \log({2{D}+a_{N,m}}) + (4 {D} + {a_{N,m}}) \log\left(\frac{4{D}+a_{N,m}}{2{D}+a_{N,m}}\right)}{2{D}+a_{N,m}}}
    \\
    &\leq  \sqrt{{2 {D} ({2{D}+a_{N,m}})\log({2{D}+a_{N,m}}) + (4 {D} + {a_{N,m}}) ({2{D}+a_{N,m}})\log\left(\frac{4{D}+a_{N,m}}{2{D}+a_{N,m}}\right)}}.
\end{align*}
The mean value theorem yields 
$
    \log\left(\frac{4{D}+a_{N,m}}{2{D}+a_{N,m}}\right)\leq  \frac{2{D}}{2{D}}=1,
$
and in consequence the following inequality
$$\int_{0}^{{2{D}+a_{N,m}}} \sqrt{ \log\left(\frac{2{D}}{\epsilon}+1\right)}d\epsilon\leq ({4{D}+a_{N,m}}) \sqrt{{2 {D} + 1}}.$$
Finally we can derive the bound 
\begin{align*}
    \sqrt{m}E\left|\sup_{f\in \mathcal{F}_c^{D}}\int f(\mathbf{y}) (dQ_m(\mathbf{y}) -dQ(\mathbf{y}) )\right|
    &\leq 8\sqrt{2}\sqrt{N} E\{({{4{D}+a_{N,m}}} ) \left(\log(2)+\sqrt{{2 {D} + 1}}\right)\}.
\end{align*}
Since, using triangle inequality, we have 
$$Ea_{N,m}=E\left(2\sqrt{\int \sup_{i=1, \dots, N}| \mathbf{y}- \mathbf{x}_i|^2dQ_m(\mathbf{y})}\right) \leq 2 \sqrt{ \int | \mathbf{y}|^2dQ(\mathbf{y})}+2D,$$ which proves the result.
\end{proof} 

The theorem provides a control on the consistency of the empirical bias for the 1-Wasserstein distance. The rate becomes slower when  $N$ the number of points defining the support of the discrete measures $P$ increases. If $P$ models an approximation of  a continuous probability on $\R^d$, hence the number $N$ required to obtain a proper approximation grows exponentially larger when the dimension $d$ increases. Hence the influence with respect to $N$ stands for the curse of dimension.

A practical  consequence of the previous bound is the following approximation problem.   Suppose that $Q$ and $P$ are a probability distributions supported on a compact set $\Omega\subset\R^d$. Assume  the probability $Q$ is unknown but observed through empirical observations giving rise to the empirical distribution $Q_m$. Let  $P$ be a known distribution which is discretized using $N$ points. Let $P^N$ be this discretization of $P$.  One aims at approximating the true 1-Wasserstein distance $\mathcal{W}_1(P,Q)$ from the empirical semi-discrete distance  $\mathcal{W}_1(P^N,Q_m)$ that can be computed. Theorem~\ref{bound_exp} and triangle inequality gives the following upper bound
\begin{equation*}
E\left|\mathcal{W}_1(P^N,Q_m)-\mathcal{W}_1(P,Q)\right| \leq    \frac{8\sqrt{2N}}{\sqrt{m}}K(\Omega, Q)+\left|\mathcal{W}_1(P^N,Q)-\mathcal{W}_1(P,Q)\right|.
\end{equation*}
We can see that there is a trade-off between the size of the sample and the size of the discretization : the first term requires $N/m$ to be small while the second term is only driven by $N$ the discretization, being smaller when the number of points is larger. 
\\
We illustrate the precision of the upper bound  with the following simulation. Consider the uniform measure on the unit interval  $U(0,1)$ and draw a  sample of size $m=2000$  to obtain the empirical $U_m$. Then from  a uniform discretization of size $N$  of the unit interval, we obtain  the discrete measure $P^{N}$.  We compute, using Monte-Carlo simulations, the empirical error $E|\mathcal{W}_1(P^N,U_m)-\mathcal{W}_1(P^N,U)|$ for different choices for $N$. The results are presented in Figure~\ref{fig:mesh1}.  We observe, in the left figure, that, for regular values of  $N$,  the growth of $E|\mathcal{W}_1(P^N,U_m)-\mathcal{W}_1(P^N,U)|$  is exactly  of order $\sqrt{N}$, following the bound. Yet for larger values of $N$ (right side) we observe that the order is no longer   $\sqrt{N}$. This is because $E|\mathcal{W}_1(P^N,U_m)-\mathcal{W}_1(P^N,U)|$ is only an upper bound and the true rate becomes smaller.

\begin{figure}[h!]
    \centering
    \includegraphics[width=0.4\textwidth]{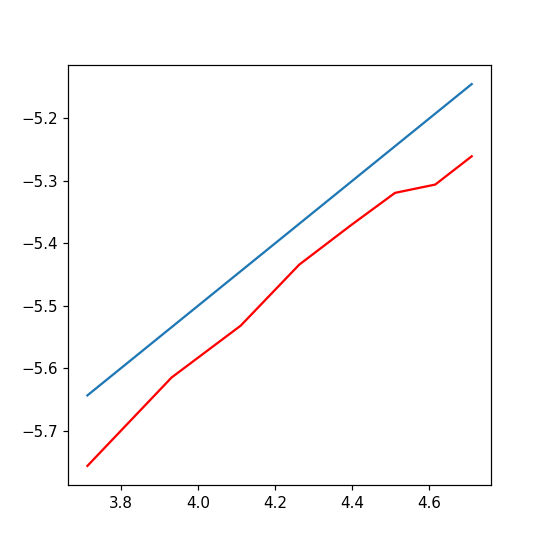}
    \includegraphics[width=0.4\textwidth]{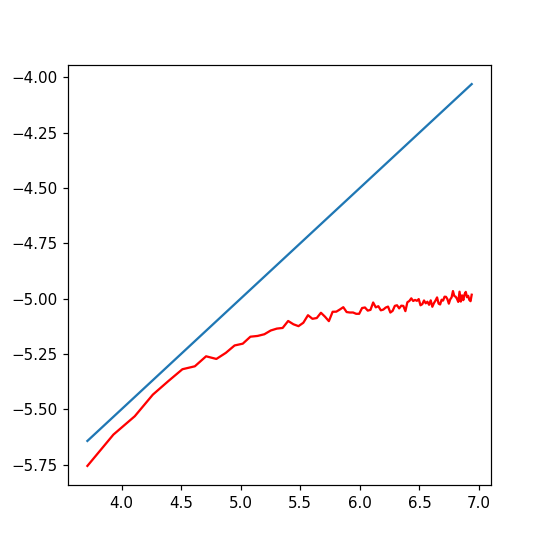}
    \caption{Plot, in double logarithmic scale, of $E|\mathcal{W}_1(P^N,U_m)-\mathcal{W}_1(P^N,U)|$ ($y$ axis) with respect to $N$ ($x$ axis). }
    \label{fig:mesh1}
\end{figure}

\section{Asymptotic Gaussian distribution optimal transport cost}
\subsection{Gaussian Limit for Empirical Optimal trandport cost}
Theorem~\ref{Teoremaprinci} is valid for generic Polish spaces. When $\mathcal{X},\mathcal{Y}$ are subsets of $\R^d$, the limit distribution in the CLT can be specified. Under  the following regularity assumptions, we prove in this section that  the limit distribution is Gaussian. \vskip .1in 
Let $Q\in \mathcal{P}(\R^d)$ be a probability measure absolutely continuous with respect to the Lebesgue measure in $\R^d$. Assume that $c(\mathbf{x},\mathbf{y})=h(\mathbf{x}-\mathbf{y})$ where $h:\R^d\rightarrow [0, \infty)$ is a non negative function satisfying: 
\begin{enumerate}
	\item[(A1):] $h$ is strictly convex on $\R^d$.
	\item[(A2):] Given a radius $r\in \R^+$ and an angle $\theta \in (0,\pi) $, there exists some $M:=M(r, \theta)>0$ such that for all $|\mathbf{p} |>M$, one can find a cone 
	\begin{align}
		K(r, \theta, \mathbf{z},\mathbf{p}):=\left\lbrace \mathbf{x}\in \R^d : | \mathbf{x}-\mathbf{p}|| \mathbf{z}|\cos(\theta/2)\leq \left< \mathbf{z},\mathbf{x}-\mathbf{p} \right>\leq  r| \mathbf{z}| \right\rbrace,
	\end{align}
	with vertex at $\mathbf{p}$ on which $h(\mathbf{x})$ attains its maximum at $\mathbf{p}$.
	\item[(A3):] $\lim_{|\mathbf{x} | \rightarrow 0}\frac{h(\mathbf{x})}{|\mathbf{x} | }= \infty $.
	%\item $h$ is differentiable and its gradient is locally Lipschitz $h\in \mathcal C^{1,1}_{loc}(\R^d)$.
\end{enumerate}
Under such assumptions, \cite{gangbo1996} shows the existence of an optimal transport map $T$ solving
\begin{align}\label{monge}
\mathcal{T}_c(P,Q):=\inf_{T}\int c(\textbf{y},T(\textbf{y})) d Q(\textbf{y}), \ \ \text{and}\ \ T_{\#}Q=P.
\end{align}
The notation $T_{\#}Q$ represents the \emph{Push-Forward} measure, defined for each measurable set $A$ by $T_{\#}Q(A):=Q(T^{-1}(A))$. The solution of \eqref{monge} is called \emph{optimal transport map} from $P$ to $Q$. Moreover it is defined as the unique Borel function satisfying 
\begin{equation}
   T(\mathbf{x})=\mathbf{x}-\nabla h^*(\nabla \varphi(\mathbf{x})), \ \ \text{where $\varphi$ solves \eqref{dual}.} 
\end{equation}
Here $h^*$ denotes the convex conjugate of $h$, see \cite{rockafellar1970}.
Such uniqueness enabled \cite{delbarrio2021central} to deduce the uniqueness, under additive constants, of the solutions of \eqref{dual} in $\varphi$. They assumed (A1)-(A3) to show that if two $c-$concave functions have the same gradient almost everywhere for $\ell_d$ in a connected open set, then both are equal, up to an additive constant. In consequence, assuming that $h$ is differentiable, the interior of the support of $Q$ is connected and with Lebesgue negligible boundary, that is, $\ell_d\left(\partial \operatorname{supp}(Q)\right)=0$, the uniqueness, up to additive constants, of the solutions of \eqref{dual} holds.  
The proof of the main theorem in this section  is a direct consequence of Lemma~\ref{Remark_uniq}, which proves that there exists an unique, up to an additive constant, $\mathbf{z}\in \operatorname{Opt}(P,Q) $.  We use within this section the notation $\mathbf{1}:=(1,\dots,1)$.
\begin{Lemma}\label{Remark_uniq}
Let $P\in \mathcal{P}(\mathbb{X})$ and $Q\in \mathcal{P}(\R^d)$ be such that $Q\ll\ell_d$ and its support is connected with Lebesgue negligible boundary. If the cost $c$ satisfies (A1)-(A3) is differentiable  and 
\begin{align*}
\int c(\mathbf{x}_i, \mathbf{y})dQ(\mathbf{y})<\infty, \ \text{ for all $i=1,\dots,N$.}
\end{align*}
Then the set $\operatorname{Opt}_c^0(P,Q)$ is a singleton.
\end{Lemma}

The following theorem states, under the previous assumptions, that the limit distribution described in Theorem~\ref{Teoremaprinci} is the centered Gaussian variable  $\left(\sqrt{\lambda} \sum_{i=1}^N z_i  X_i+(\sqrt{1-\lambda})\mathbb{G}^c_Q(\mathbf{z})\right)$ where $\{\mathbf{z}\}= \operatorname{Opt}_c^0(P,Q)$.
Note that  $\sum_{i=1}^N z_i  X_i$  is Gaussian and centered, with variance
\begin{equation}\label{definition_gaussian}
 \sigma^2(P,\mathbf{z})=\operatorname{Var}( \sum_{i=1}^N z_i  X_i)\ \ \text{and } \ \ (X_1,\dots,X_N)\sim\mathcal{N}(\mathbf{0},\Sigma(\mathbf{p})),
\end{equation}
where $\Sigma(\mathbf{p})$ is defined in \eqref{sigma_Natrix}. On the other side $\mathbb{G}^c_Q(\mathbf{z})$ follows the distribution $\mathcal{N}(0,\sigma^2_c(Q,\mathbf{z})))$, where 
\begin{equation}\label{definition_gaussiancont}
\sigma^2_c(Q,\mathbf{z})=\int \inf_{i=1, \dots, N} \{ c(\mathbf{x}_i, \mathbf{y} )-z_i\}^2dQ(\mathbf{y})\\
-\int \inf_{i=1, \dots, N} \{ c(\mathbf{x}_i, \mathbf{y} )-z_i\}^2dQ(\mathbf{y}).
\end{equation}
Since, for every $\lambda\in \R$,  we have that $ \sigma^2(P,\mathbf{z})= \sigma^2(P,\mathbf{z}+\lambda\mathbf{1})$, then the asymptotic variance obtained in the following theorem is well defined.
\begin{Theorem}\label{theo:semi_disc-p}
Let $P\in \mathcal{P}(\mathbb{X})$ and $Q\in \mathcal{P}(\R^d)$ be such that $Q\ll\ell_d$ and its support is connected with Lebesgue negligible boundary. If the cost $c$ satisfies (A1)-(A3) is differentiable and 
\begin{align*}
\int c(\mathbf{x}_i, \mathbf{y})dQ(\mathbf{y})<\infty, \ \text{ for all $i=1,\dots,N$.}
\end{align*}
then the following limits hold.
\begin{itemize}
    \item \textbf{(One sample case for $P$)}
    \begin{align*}
	\sqrt{n}\left(\mathcal{T}_c(P_n,Q)- \mathcal{T}_c(P,Q)\right)\xrightarrow{w} X\sim \mathcal{N}(0,\sigma^2(P,\mathbf{z})).
\end{align*}
\end{itemize}
Suppose that
\begin{align}\label{cuadratic2}
\int c(\mathbf{x}_i, \mathbf{y})^2dQ(\mathbf{y})<\infty, \ \text{ for all $i=1,\dots,N$}.
\end{align} 
\begin{itemize}
 \item \textbf{(0ne sample case for $Q$)}
$$\sqrt{m}\left(\mathcal{T}_c(P,Q_m)- \mathcal{T}_c(P,Q)\right)\stackrel{w}\longrightarrow Y\sim \mathcal{N}(0,\sigma^2_c(Q,\mathbf{z})).$$
\item  \textbf{(Two sample case )}
if $n,m\rightarrow\infty$, with $\frac{m}{n+m}\rightarrow \lambda\in  (0,1)$,  then
$$\sqrt{\frac{nm}{n+m}}\left(\mathcal{T}_c(P_n,Q_m)- \mathcal{T}_c(P,Q)\right)\stackrel{w}\longrightarrow  \sqrt{\lambda}X+(\sqrt{1-\lambda})Y.$$
\end{itemize}
Here, $ \sigma^2(P,\mathbf{z})$  and $ \sigma^2_c(Q,\mathbf{z})$ are defined in  \eqref{definition_gaussian} and  \eqref{definition_gaussiancont} and, moreover, $X$ and $Y$ are independent.
\end{Theorem}

As in  the previous section, we provide an application to the CLT for Wasserstein distances. The potential costs $c_p=|\cdot|^p$, for $p>0$, satisfy (A1)-(A3), then the following result follows immediately from Theorem~\ref{theo:semi_disc-p} and the Delta-Method for the function $t\mapsto |t|^{\frac{1}{p}}$. Recall that, in the potential cost cases, $\mathcal{T}_p(P,Q)$ denotes the optimal transport cost and $\mathcal{W}_p(P,Q)=\left(\mathcal{T}_p(P,Q)\right)^{\frac{1}{p}}$ the $p$-Wasserstein distance.
\begin{Corollary}\label{Coro:semi_disc-p}
Let $P\in \mathcal{P}(\mathbb{X})$ be as in \eqref{represen} and $Q\in \mathcal{P}(\R^d)$ be such that $Q\ll\ell_d$, has finite moments of order $p$ and its support is connected with Lebesgue negligible boundary.
Then, for every $p>1$, we have that
\begin{itemize}
    \item \textbf{(One sample case for $P$)}
\begin{align*}
	\sqrt{n}\left(\mathcal{T}_p(P_n,Q)- \mathcal{T}_p(P,Q)\right)\xrightarrow{w} \mathcal{N}(0,\sigma^2(P,\mathbf{z})),
\end{align*}
and 
\begin{align*}
	\sqrt{n}\left(\mathcal{W}_p(P_n,Q)- \mathcal{W}_p(P,Q)\right)\stackrel{w}\longrightarrow \mathcal{N}\left(0,\left(\frac{1}{p\left( \mathcal{W}_p(P,Q)\right)^{p-1}}\right)^2\sigma^2(P,\mathbf{z})\right),
\end{align*}
\end{itemize}
Suppose that $Q$ has finite moments of order $2p$,
then
\begin{itemize}
 \item \textbf{(0ne sample case for $Q$)}
 \begin{align*}
	\sqrt{m}\left(\mathcal{T}_p(P,Q_m)- \mathcal{T}_p(P,Q)\right)\xrightarrow{w} \mathcal{N}(0,\sigma^2_{p}(Q,\mathbf{z})),
\end{align*}
and 
\begin{align*}
	\sqrt{m}\left(\mathcal{W}_p(P,Q_m)- \mathcal{W}_p(P,Q)\right)\stackrel{w}\longrightarrow \mathcal{N}\left(0,\left(\frac{1}{p\left( \mathcal{W}_p(P,Q)\right)^{p-1}}\right)^2\sigma^2_p(Q,\mathbf{z})\right),
\end{align*}
\item  \textbf{(Two sample case)}
if $n,m\rightarrow\infty$, with $\frac{m}{n+m}\rightarrow \lambda\in  (0,1)$,  then
 \begin{align*}
\sqrt{\frac{nm}{n+m}}\left(\mathcal{T}_p(P_n,Q_m)- \mathcal{T}_p(P,Q)\right)\overset{w}\longrightarrow \mathcal{N}(0,\lambda \sigma^2(P,\mathbf{z})+(1-\lambda)\sigma^2_{p}(Q,\mathbf{z})),
\end{align*}
and 
 \begin{align*}
\sqrt{\frac{nm}{n+m}}\left(\mathcal{W}_p(P_n,Q_m)- \mathcal{W}_p(P,Q)\right)\overset{w}\longrightarrow \mathcal{N}\left(0,\left(\frac{1}{p\left( \mathcal{W}_p(P,Q)\right)^{p-1}}\right)^2\left(\lambda \sigma^2(P,\mathbf{z})+(1-\lambda)\sigma^2_{p}(Q,\mathbf{z})\right)\right).
\end{align*}
\end{itemize}
Here, $\sigma^2(P,\mathbf{z})$ and $\sigma^2_{p}(Q,\mathbf{z})$ are defined in \eqref{definition_gaussian} and \eqref{definition_gaussiancont}
for $\mathbf{z}\in \operatorname{Opt}_{|\cdot|^p}(P,Q)$ and the cost $|\cdot|^p$.
\end{Corollary}
Since $P$ is discrete and $Q$ is continuous, $\mathcal{W}_p(P,Q)>0$ and the limit distribution of Corollary~\ref{Coro:semi_disc-p} is always well defined.

Note that Corollary~\ref{Coro:semi_disc-p} is a particular case of Corollary~\ref{coro:semi_disc} in the cases where the optimal transport potential is unique---the hypotheses of Theorem~\ref{theo:semi_disc-p} hold---which is the reason why the case $p=1$ can not be considered. Concerning other potential costs, $p>1$,  it is straightforward to see that the hypotheses (A1)-(A3) hold, see for instance \cite{delbarrio2021central} or  \cite{gangbo1996}. 
\subsection{Simulations}
This  section is devoted to illustrate empirically Theorems~\ref{Teoremaprinci} and \ref{theo:semi_disc-p}. The limit distribution depends on the true Wasserstein distance between the distributions. Hence to simulate the central limit theorems, the difficulty lies in proving the consistency of its bootstrap approximation. Actually the non fully Hadamard differentiability of the functional implies that  the limit in Theorem~\ref{Teoremaprinci} is the supremum of Gaussian processes. In consequence, as pointed out in \cite{Santos}, the bootstrap will not be consistent. However,  in the framework of Theorem~\ref{theo:semi_disc-p}, the dual problem has a unique solution. In consequence, the mapping is fully Hadamard differentiable (Corollary~2.4 in \cite{Luis2020}) which implies that the bootstrap procedure is consistent (\cite{Santos}). This enables us to approximate the variance as shown in the following simulations. \vskip .1in

Here we implement one favorable case for bootstrap approximation. In particular we choose the quadratic cost $|\cdot |^2$ and the discrete probability $P=\frac{1}{7}\sum_{i=1}^7\delta_{\mathbf{x}_i}$, where $$ \mathbb{X}=\{\mathbf{x}_i\}_{i=1}^7=\{(1,0,0)  , (0,1,0)  ,(0,0,1) , (-1,0,0) , (0,-1,0) ,(0,0,-1) ,(0,0,0) \}. $$ 
The continuous probability $Q\in \mathcal{P}(\R^3)$ is the direct product  $\mathcal{U}(-1,1)\times\mathcal{N}(0,1)\times\mathcal{N}(0,1)$. 
\begin{figure}[h!]
   \centering
    \includegraphics[width=0.4\textwidth]{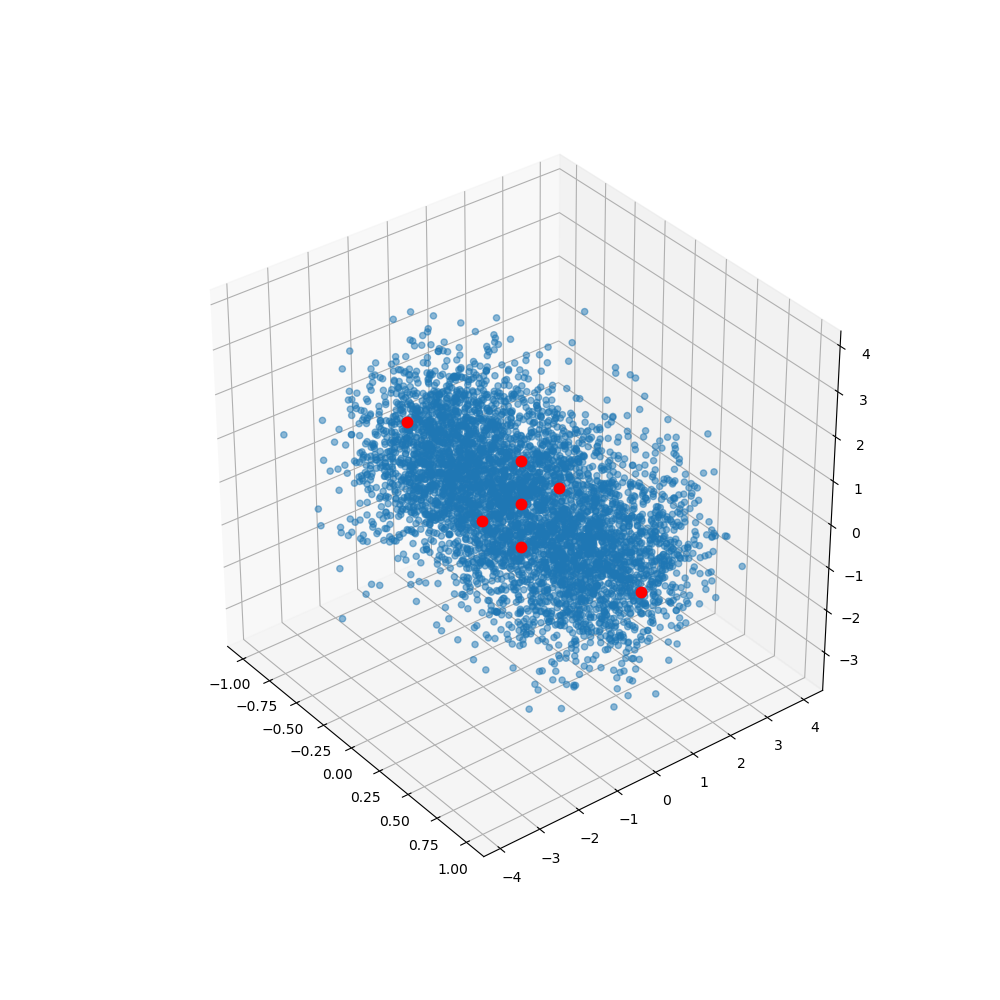}
     \includegraphics[width=0.4\textwidth]{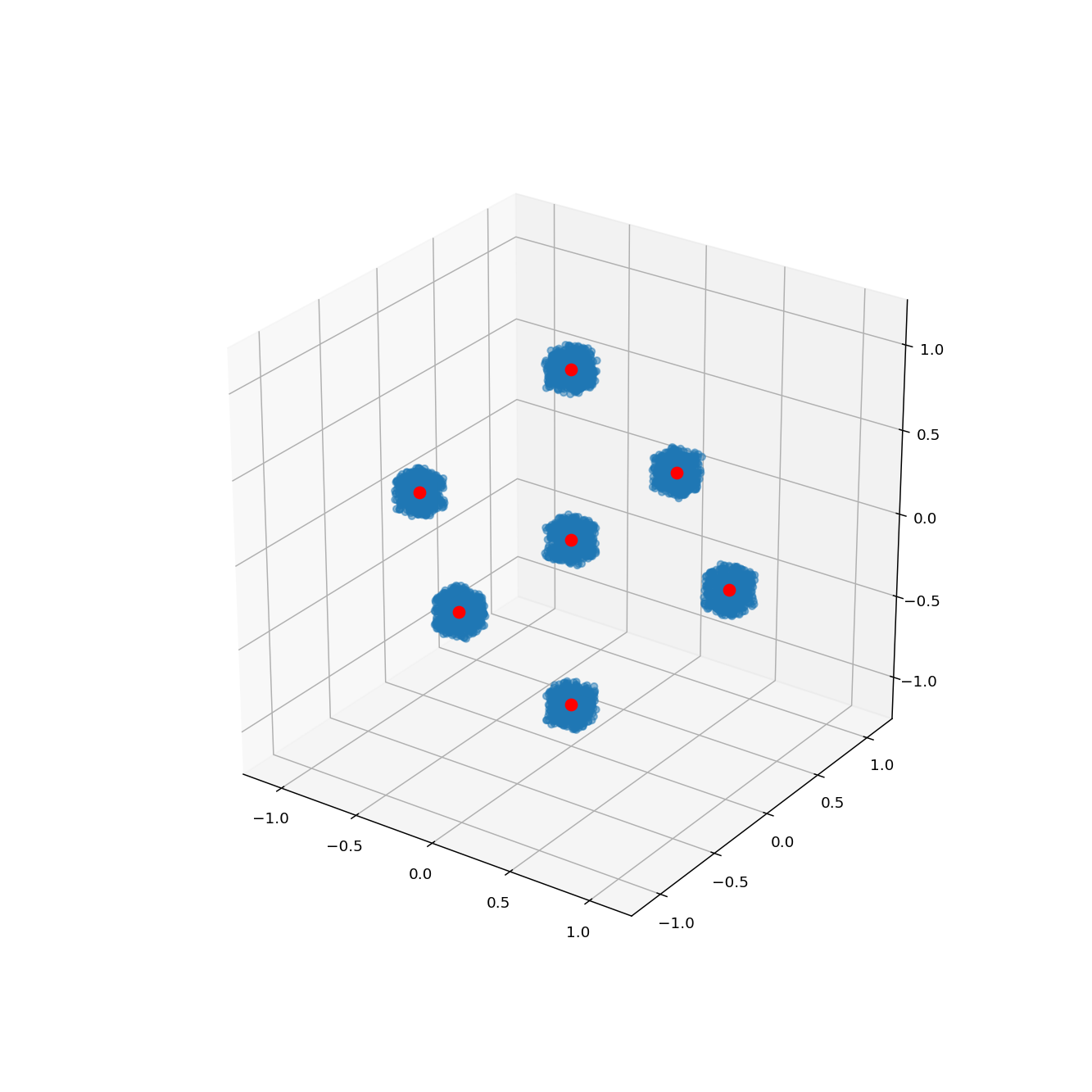}
    \caption{3D visualization of the data set. In blue the continuous distribution $Q$ and in red the discrete one $P$. Left: $Q$ has a density with connected  support and  satisfies the assumptions of Theorem~\ref{theo:semi_disc-p}. Right:  $Q$ has a density but not a connected  support. }
    \label{fig:3dvisu}
\end{figure}
Note that its support is connected with Lebesgue negligible boundary---we can visualize the data in Figure~\ref{fig:3dvisu}---and satisfies  the assumptions of Theorem~\ref{theo:semi_disc-p}. As commented before, we can use the bootstrap procedure. In this example, it is assumed that the discrete $P$ is known and the sample, of size $m=5000$, comes from the continuous $Q$. Figure~\ref{fig:CLTBootstrap} shows the result of the bootstrap procedure for a re-sampling size of $10000$. The simulations follow the asymptotic theory we provide.
\begin{figure}[h!]
   \centering
    \includegraphics[width=0.4\textwidth]{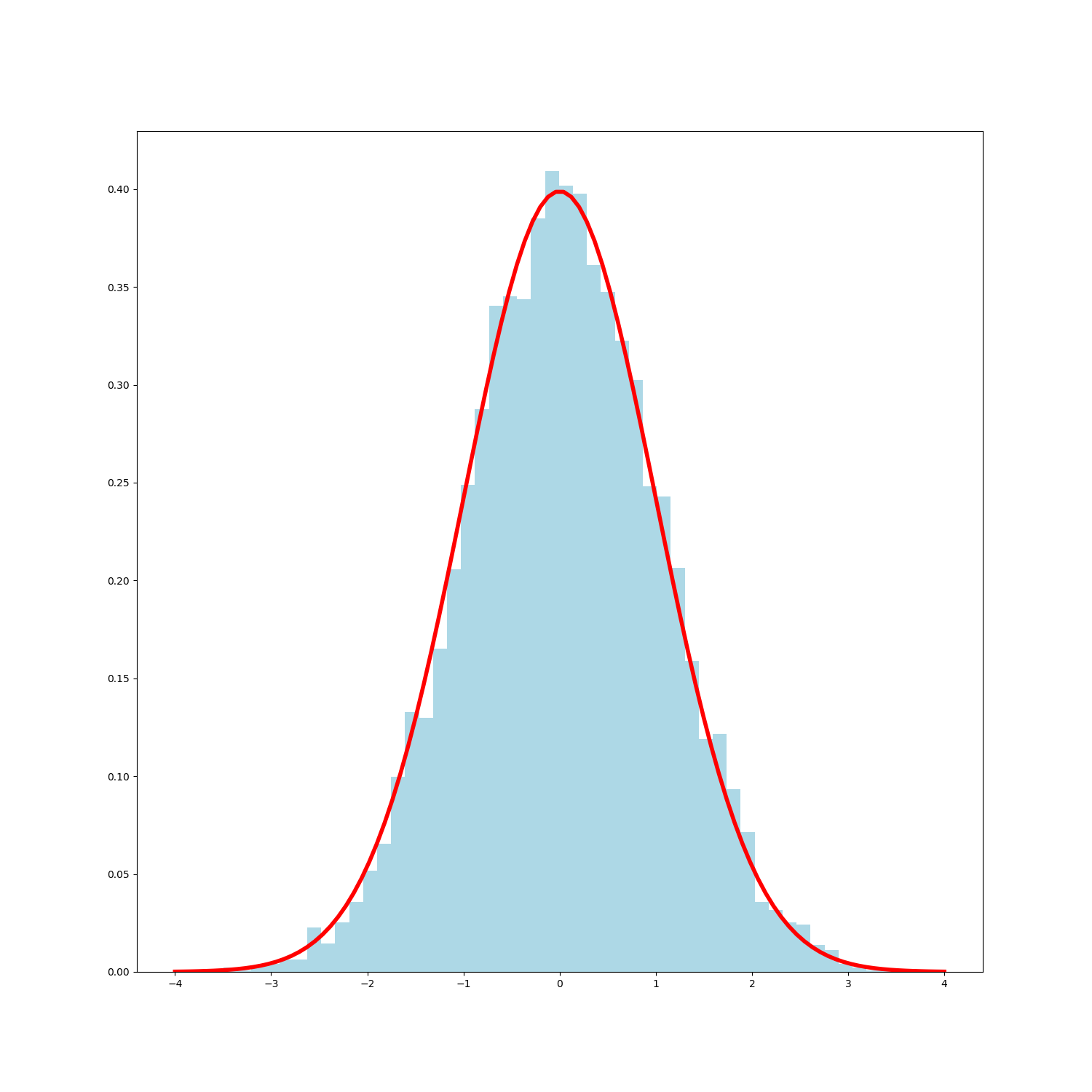}
    \includegraphics[width=0.4\textwidth]{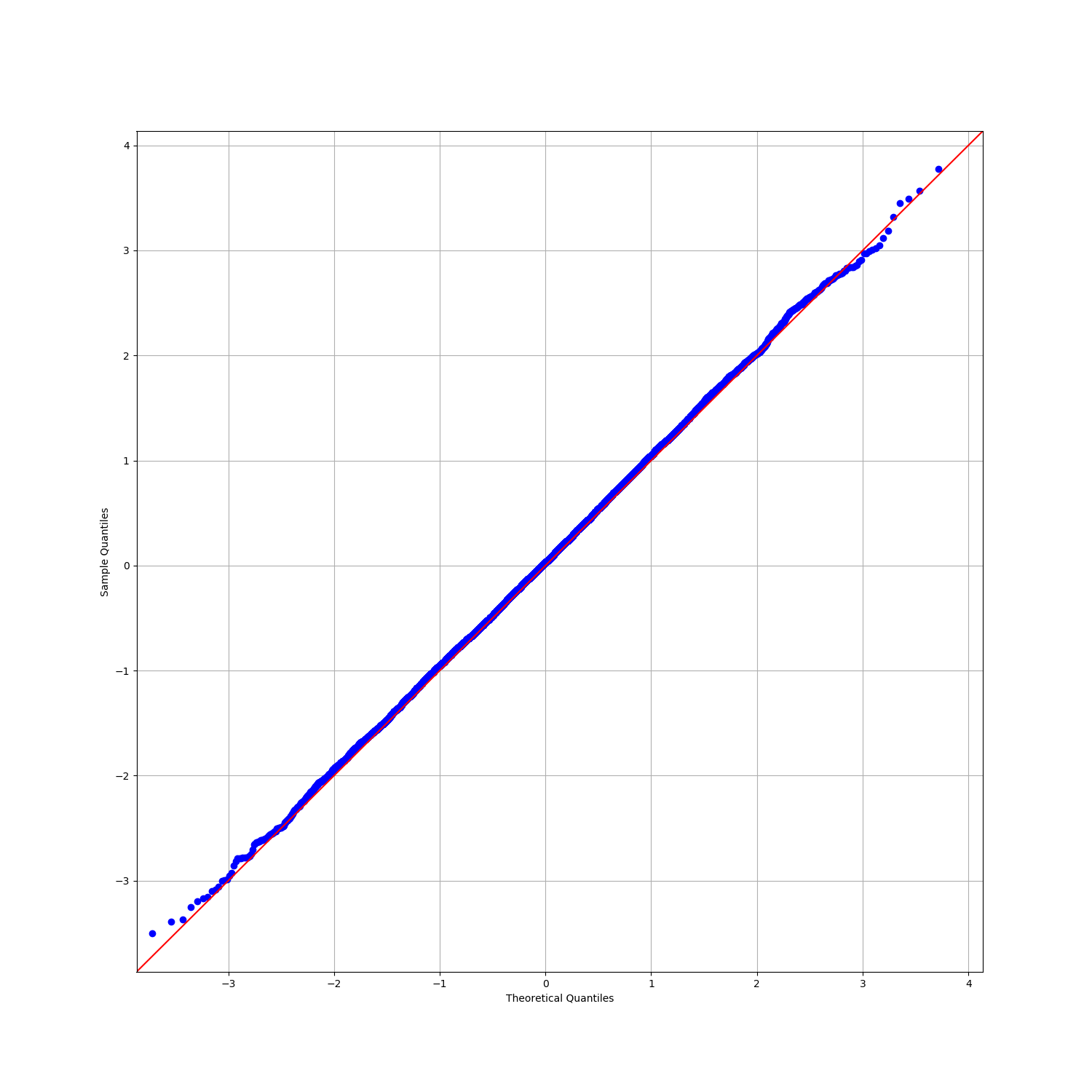}
    \caption{Illustration of Theorem~\ref{theo:semi_disc-p} using bootstrap procedures. Histograms  (left) and Q–Q plot (right) of the bootstrap estimation of $\sqrt{m}\frac{\left(\mathcal{W}_2^2(P,Q_m)-\mathcal{W}_2^2(P,Q)\right)}{\sigma_{2}(Q_m,\mathbf{z}^m)}$ .}
    \label{fig:CLTBootstrap}
\end{figure}
\\

Now we illustrate a case where the assumptions of Theorem~\ref{theo:semi_disc-p} are no longer fulfilled. More precisely, we consider $Q$ as the continuous probability with density
$\frac{1}{0.008\cdot 7}\sum_{i=1}^7 \mathbbm{1}_{\textbf{x}_i+(-0.1,0.1)^3}$, ---this is a mixture model of uniform probabilities on small cubes centered in the points of $\mathbb{X}$---we can see a 3d plot in Figure~\ref{fig:3dvisu}. To approximate the limit distribution we need first to estimate the value $\mathcal{W}_2^2(P,Q)$. We make it by an independent sample of size $10000$ and computing the mean by Monte Carlo $100$ times. Then we compute the histogram of $\sqrt{m}\frac{\left(\mathcal{W}_2^2(P,Q_m)-\mathcal{W}_2^2(P,Q)\right)}{\sigma_{2}(Q_m,\mathbf{z}^m)} $ with the original sample. The results are shown in Figure~\ref{fig:Monte}, we can see, clearly,  that  the limit is not Gaussian. Similar examples with non-Gaussian limits can be found in Figure~1 in \cite{Sommerfeld2018}. But Figure~\ref{fig:Monte} is quite different from their experimentation since one of the probabilities is continuous and \cite{Sommerfeld2018} studies only the  optimal transport problem between discrete probabilities.
\begin{figure}[h!]
   \centering
    \includegraphics[width=0.5\textwidth]{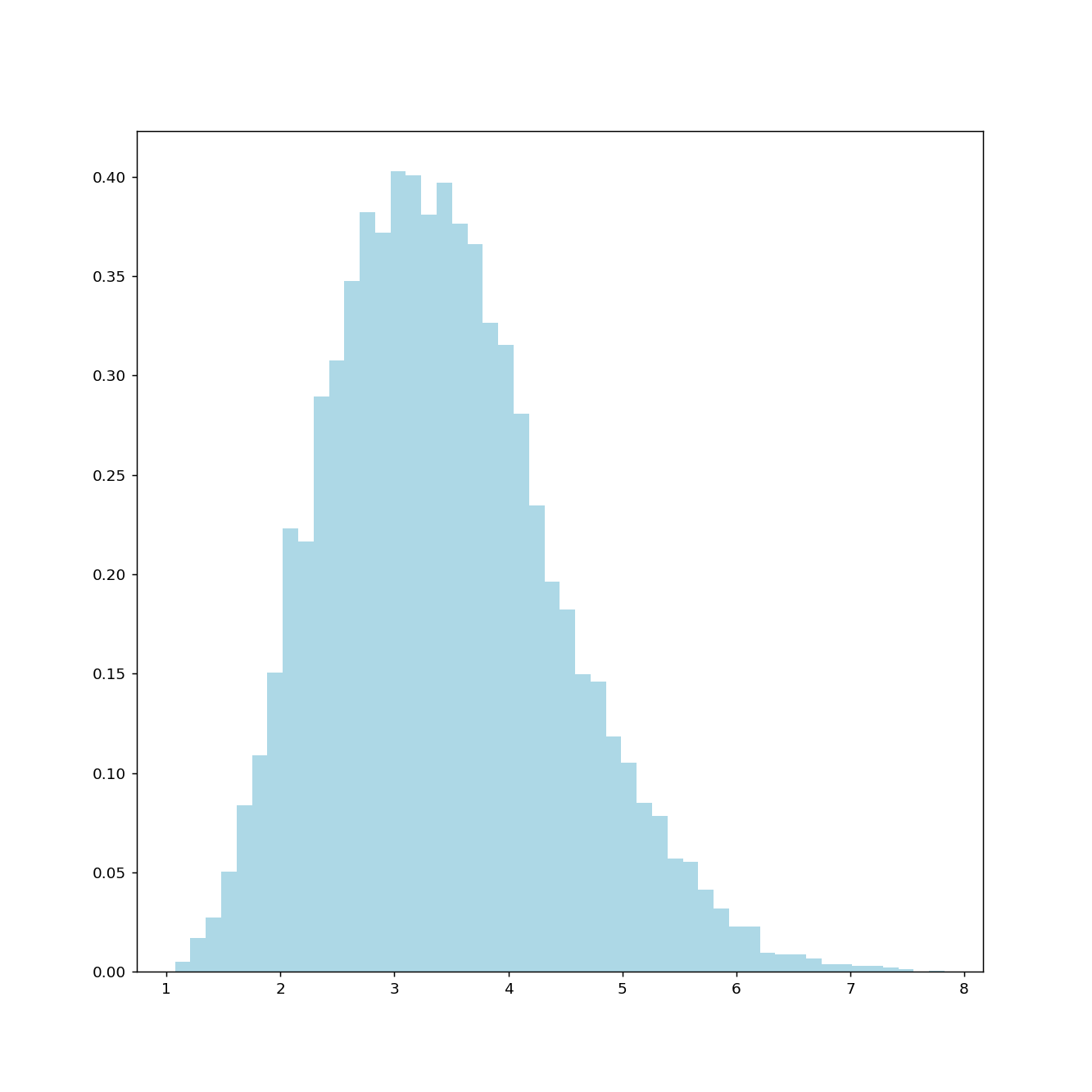}
    
    \caption{Illustration of Theorem~\ref{Teoremaprinci}, by using Monte Carlo's method, for $Q$ with disconnected support.}
    \label{fig:Monte}
\end{figure}

\section{A central Limit theorem for the potentials.}
The aim of this section is to provide a CLT for the empirical potentials, defined as the solutions to the empirical version of  the dual formulation of the Monge-Kantorovich problem \eqref{dual}. In the semidiscrete case the potentials are pairs formed by $\mathbf{z}=(z_1,\dots,z_N)\in \operatorname{Opt}_c(P,Q)$ and  $
    \label{phi}\varphi(\mathbf{y}):=\inf_{i=1, \dots,N} \{ c(\mathbf{x}_i, \mathbf{y} ) -{z_i} \}.
$
Note that potentials are defined up to a constant in the sense that if $(\psi,\varphi)$ solves \eqref{dual} then $(\psi+C,\varphi-C)$ also  solves \eqref{dual}, for any constant $C$.  Hence we will  study  the properties of the following functional, defined in  $\langle \mathbf{1}\rangle^{\perp} $ which denotes the orthogonal complement of the vector space generated by $\mathbf{1}=(1,\dots, 1)$
\begin{align*}
    \mathcal{M}_{\mathbf{p}}:\langle \mathbf{1}\rangle^{\perp} &\longrightarrow \R \\
    \mathbf{z}&\mapsto  g_c(P,Q,\mathbf{z}),
\end{align*}
where $g_c(P,Q,\mathbf{z})$ is defined as in \eqref{defg}.

%The idea behind the proof of the main result of the section, Theorem~\ref{Theo:potential},  is to rescue some of the classical  results in $M$-estimation. In consequence, the strictly positiveness of the Hessian of the previously defined $\Gamma$ is required. Such study of the second derivative have already been addressed in \cite{JUN2019}, where they used a completely different notation,  more related to the angle of mathematical analysis than the statistical. 
In this section we will use some framework developed in~\cite{JUN2019}.  So in  this section we propose some slight changes of the notations yet maintaining as much coherence as possible with  the previous ones.\\
\indent  First we will assume that $\mathcal{Y}$ is an open domain of a $d$-dimensional Riemannian manifold $\mathcal{R}$ endowed with the volume measure $\mathcal{V}$ and metric $d$.  We consider  $\mathcal{C}(\mathcal{Y})$, $\mathcal{C}^1(\mathcal{Y})$ and $\mathcal{C}^{1,1}(\mathcal{Y})$  the spaces of real valued continuous functions, real valued continuously differentiable functions  and  the space of real valued continuously differentiable functions with  Lipschitz derivatives, respectively. \\ \\
Following the approach in \cite{JUN2019}, we assume  that the cost satisfies the  following assumptions
\begin{align}
    \label{Reg}
    \tag{Reg}
    \text{$c(\mathbf{x}_i,\cdot)\in \mathcal{C}^{1,1}(\mathcal{Y})$, for all $i=1,\dots,N$,}
\end{align}
\begin{align}
        \label{Twist}
    \tag{Twist}
    \text{$D_{\mathbf{y}}c(\mathbf{x}_i,\mathbf{y}):\mathcal{Y}\rightarrow T^*_{\mathbf{y}}(\Omega)$ is injective as a function of $\mathbf{y}$, for all $i=1,\dots,N$},
\end{align}
where $D_{\mathbf{y}}c$ denotes the partial derivative of $c$ w.r.t. the second variable.
    For every $i\in \{1, \dots,N\}$ there exists $\Omega_i\subset \R^d$ open and convex set, and a $\mathcal{C}^{1,1}$ diffeomorphism $\exp^c_{i}:\Omega_i\rightarrow \Omega$ such that the functions
\begin{align}
    \label{QC}
    \tag{QC}
        \text{$\Omega_{i}\ni\mathbf{p}\mapsto f_{i,j}(\mathbf{p}):=c(\mathbf{x}_i,\exp^c_{i}\mathbf{p})-c(\mathbf{x}_j,\exp^c_{i}\mathbf{p})$  are quasi-convex for all $j=1,\dots,N$.}
\end{align}
Here quasi-convex, according to \cite{JUN2019}, means that for every $\lambda\in \R^d$ the sets $f_{i,j}^{-1}([-\infty, \lambda])$ are convex. \\
Besides the assumptions on the cost, we assume that  the probability is supported in a \emph{$c$-convex set} $\Omega$, which  means that $ (\exp^c_{i})^{-1}(\Omega)$ is convex, for every $i=1,\dots,N$. Formally,
let $\mathcal{Y}\subset \mathcal{R}$ be a compact $c$-convex set, $P\in \mathcal{P}(\mathbb{X})$ be as in \eqref{represen} and suppose that
  \begin{equation}
 \label{Hol}
     \tag{Cont}
     \text{$Q\in \mathcal{P}(\mathcal{Y})$ satisfies $Q\ll\mathcal{V}$ with density $q\in \mathcal{C}(\mathcal{Y})$}.
 \end{equation}
The last required assumption in  \cite{JUN2019} is that $Q$ satisfies a \emph{Poincar\'e-Wirtinger inequality with constant} $C_{PW}$:  a probability measure $Q$ supported in a compact set $\mathcal{Y}\subset\mathcal{R}$ satisfies a Poincar\'e-Wirtinger inequality with constant $C_{PW}$ if for every $f\in \mathcal{C}^1(\mathcal{Y})$ we have that for $Y\sim Q$
 \begin{equation}
 \label{PW}
     \tag{PW}
     E(|f(Y)-E(f(Y))|)\leq C_{PW} E(|\nabla f(Y)|).
 \end{equation} 
  In order to clarify the feasibility of such assumptions, we will provide some insights on them at the end of the section.
 \cite{JUN2019} proved the following assertions.  
 \begin{enumerate}
    \item \label{item_Lemma.1} Under assumptions \eqref{Reg} and \eqref{Twist} the function $\mathcal{M}_{\mathbf{p}}(\mathbf{z})$ is concave and differentiable with derivative
 $$\nabla_{\mathbf{z}} \mathcal{M}_{\mathbf{p}}(z) = (-Q(A_2(\mathbf{z}))+p_2, \dots, -Q(A_N(\mathbf{z}))+p_N ),$$
 where 
\begin{align}\label{A_k}
A_k(\mathbf{z}):=\{ \mathbf{y} \in \R^d :\ \ c(\mathbf{x}_k,\mathbf{y} ) -z_k <c(\mathbf{x}_i,\mathbf{y} )-z_i, \ \ \text{for all $i\neq k$} \}.
\end{align}
\item \label{item_Lemma.3} Under assumptions \eqref{Reg},\eqref{Twist} and \eqref{QC}, the function $\mathcal{M}_{\mathbf{p}}$ is twice continuously differentiable with    Hessian matrix $D_{\mathbf{z}}^2 \mathcal{M}_{\mathbf{p}}(\mathbf{z})=\left(\frac{\partial^2 \mathcal{M}_{\mathbf{p}}}{\partial z_i\partial z_j}(\mathbf{z})\right)_{i,j=1,\dots,N}$  and partial derivatives 
\begin{align}
\begin{split}\label{Hessian}
 \frac{\partial^2 }{\partial z_i\partial z_j}  \mathcal{M}_{\mathbf{p}}(\mathbf{z}))&=\int_{A_k(\mathbf{z})\cap A_k(\mathbf{z}) } \frac{1}{|\nabla_{\mathbf{y}} c(\mathbf{x}_i,\mathbf{y})-\nabla_{\mathbf{y}} c(\mathbf{x}_j,\mathbf{y})|}dQ(\mathbf{y}), \ \ \text{if  $i\neq j$,}\\
\frac{\partial^2}{\partial^2 z_i} \mathcal{M}_{\mathbf{p}}(\mathbf{z}) &=-\sum_{j\neq i}\frac{\partial^2 }{\partial z_i\partial z_j} \mathcal{M}_{\mathbf{p}}(\mathbf{z}).  
\end{split}
\end{align}
\item Under assumptions \eqref{Reg},\eqref{Twist} and \eqref{QC}, and if $Q$ satisfies \eqref{PW}, then there exists a constant $C$ such that 
\begin{equation}\label{estrict_epsilon}
    \text{$\langle D^2_{\mathbf{z}}  \mathcal{M}_{\mathbf{p}}(\mathbf{z}) \mathbf{v}, \mathbf{v} \rangle \leq -C \epsilon ^3 |\mathbf{v}|^2$, for all $\mathbf{z}\in \mathcal{K}^{\epsilon}$ and $\mathbf{v}\in \langle \mathbf{1}\rangle^{\perp}$,}
\end{equation}
where 
$$ \mathcal{K}^{\epsilon}:=\{\mathbf{z}\in \R^d: \ Q(A_i(\mathbf{z}))>\epsilon, \ \text{ for all $i=1, \dots,N$.} \}.$$
\end{enumerate}

These previous results imply immediately the following Lemma.
\begin{Lemma}\label{Lemma_derivatve2} 
Let $\mathcal{Y}\subset \mathcal{R}$ be a compact $c$-convex set, $P\in \mathcal{P}(\mathbb{X})$ and   $Q\in \mathcal{P}(\mathcal{Y})$ . Under assumptions \eqref{Reg}, \eqref{Twist} and \eqref{QC} on the cost $c$ and \eqref{PW} and \eqref{Hol} on $Q$, we have that the function $\mathcal{M}_{\mathbf{p}}$ is strictly concave and twice continuously differentiable, with 
\begin{align*}
 \nabla \mathcal{M}_p(\mathbf{z})&=\nabla_{\mathbf{z}} \mathcal{M}_{\mathbf{p}}(\mathbf{z}) |_{\langle \mathbf{1}\rangle^{\perp}}, \\
  D^2 \mathcal{M}_p(\mathbf{z})&=D^2_{\mathbf{z}} \mathcal{M}_{\mathbf{p}}(\mathbf{z}) |_{\langle \mathbf{1}\rangle^{\perp}}.
\end{align*}
Moreover if $\bar{\mathbf{z}}\in \langle \mathbf{1}\rangle^{\perp}\cap  \operatorname{Opt}_c(P,Q)$, then there exists a constant $C$ such that 
\begin{equation}\label{Lemma_estrict}
    \text{$\langle D^2\mathcal{M}_{\mathbf{p}}(\bar{\mathbf{z}}) \mathbf{v}, \mathbf{v} \rangle \leq -C \inf_{i=1, \dots,N}|p_i| ^3 |\mathbf{v}|^2$, for all $\mathbf{v}\in \langle \mathbf{1}\rangle^{\perp}$.}
\end{equation}
\end{Lemma}
\begin{proof}
Note that it only remains to prove that \eqref{Lemma_estrict} holds.  But this is a direct consequence of \eqref{estrict_epsilon}. In fact, since $\bar{\mathbf{z}}$ is the unique $\mathbf{z}\in \operatorname{Opt}(P,Q)$, then $Q(A_k(\bar{\mathbf{z}}))=p_k$ for $k=1,\dots,N$ and we can conclude.
\end{proof}
%The sets $A_k(\mathbf{z})$ of Lemma~\ref{Lemma_derivatve2} are commonly called \emph{Laguerre cells}, see \cite{LEVY2018135}.
Now we can formulate the main theorem of this section which yields a central limit theorem for the empirical estimation of the potentials. We follow classical arguments of $M$-estimation by writing the function $\mathcal{M}_{\mathbf{p}}(\mathbf{z})$ as $ E(g(X, \mathbf{z}))$ with $X\sim P$ and $g:\mathbb{X}\times \langle \mathbf{1}\rangle^{\perp}\rightarrow \R$ defined by 
\begin{align}\label{g}
        g(\mathbf{x}_k, \mathbf{z})&:= z_k+\int \inf_{i=1, \dots,N} \{ c(\mathbf{x}_i, \mathbf{y} ) -z_i \}dQ(\mathbf{y}),
\end{align}
for each $\mathbf{z}=(z_1,\dots,z_N)$. The weak limit is a centered multivariate Gaussian distribution with covariance matrix, depending on the optimal $\tilde{\mathbf{z}}$, defined by the map 
\begin{equation}\label{eq:sigmaZ}
\Sigma(\tilde{\mathbf{z}})=
\left(D^2\mathcal{M}_{\mathbf{p}}(\tilde{\mathbf{z}} )\right)^{-1}A
    \left(D^2\mathcal{M}_{\mathbf{p}}(\tilde{\mathbf{z}} )\right)^{-1},\ \text{
where $A=\sum_{i=1}^N p_i \nabla_{\tilde{\mathbf{z}}}g(\mathbf{x}_i, \mathbf{z}) \nabla_{\tilde{\mathbf{z}}}g(\mathbf{x}_i, \mathbf{z})^{t}.$}
\end{equation}

\begin{Theorem}\label{Theo:potential}
Let $\mathcal{Y}\subset \mathcal{R}$ be a compact $c$-convex set, $P\in \mathcal{P}(\mathbb{X})$ and $Q\in \mathcal{P}(\mathcal{Y})$. Under Assumptions \eqref{Reg}, \eqref{Twist} and \eqref{QC} on the cost $c$, and \eqref{PW} and \eqref{Hol} on $Q$, we have that
\begin{equation}
    \sqrt{n}\left(\hat{\mathbf{z}}_n-\tilde{\mathbf{z}}\right)\stackrel{w}\longrightarrow N(\mathbf{0},\Sigma(\tilde{\mathbf{z}})),
\end{equation}
where $\tilde{\mathbf{z}}\in \langle \mathbf{1}\rangle^{\perp}\cap \operatorname{Opt}_c(P,Q)$ (resp. $\hat{\mathbf{z}}_n\in \langle \mathbf{1}\rangle^{\perp}\cap \operatorname{Opt}_c(P_n,Q)$), and 
$\Sigma(\tilde{\mathbf{z}})$ is defined in \eqref{eq:sigmaZ}.
\end{Theorem}
\begin{proof}
Now let $g:\mathbb{X}\times \langle \mathbf{1}\rangle^{\perp}\rightarrow \R$ be defined in \eqref{g}. It satisfies that if $X\sim P$ then $ E(g(X, \mathbf{z}))=\mathcal{M}_{\mathbf{p}}(\mathbf{z})$.
Lemma~\ref{Lemma_derivatve2} implies in particular that:
\begin{enumerate}[(i)]
    \item the function $\mathbf{z}\mapsto g(\mathbf{x}_k, \mathbf{z})$ is concave for every $\mathbf{x}_k$ .
    \item There exists a unique $$\tilde{\mathbf{z}}\in \arg\sup_{
    \mathbf{z}\in \langle \mathbf{1}\rangle^{\perp}} E(g(X, \mathbf{z})).$$
    \item The empirical potential is defined as $$\hat{\mathbf{z}}_n\in \arg\sup_{
    \mathbf{z}\in \langle \mathbf{1}\rangle^{\perp}} \int g(\mathbf{x}, \mathbf{z})dP_n(\mathbf{x}).$$
     \item The function $\mathbf{z}\mapsto E(g(X, \mathbf{z}))$ is twice differentiable in $\tilde{\mathbf{z}}$ with strictly negative definite Hessian matrix  $D^2\mathcal{M}_{\mathbf{p}}(\tilde{\mathbf{z}} )$.
     \item For every $\mathbf{z}\in \R^{m-1}$ we have that  $$|\nabla_{\mathbf{z}}g(\mathbf{x}_k, \mathbf{z})|\leq |(Q(A_2(\mathbf{z}))+1, \dots, Q(A_N(\mathbf{z}))+1 )|\leq \sqrt{2(m-1)}.$$
     
\end{enumerate}
Then all the assumptions of Corollary 2.2 in \cite{Huang2007CENTRALLT} are satisfied by the function $g$. As a consequence, we have the limit
\begin{equation}
    \sqrt{n}\left(\hat{\mathbf{z}}_n-\tilde{\mathbf{z}}\right)\xrightarrow[]{w} N\left(\mathbf{0},\left(D^2\mathcal{M}_{\mathbf{p}}(\tilde{\mathbf{z}} )\right)^{-1}A
    \left(D^2\mathcal{M}_{\mathbf{p}}(\tilde{\mathbf{z}} )\right)^{-1}\right),
\end{equation}
where $A=E\left( \nabla_{\mathbf{z}}g(X, \mathbf{z}) \nabla_{\mathbf{z}}g(X, \mathbf{z})^{t}\right).$ Note that computing $A$ we obtain the expression  \eqref{eq:sigmaZ}.
\end{proof}
For $\tilde{\mathbf{z}}$ defined as in Theorem~\ref{Theo:potential}, set
\begin{equation}\label{c-conj}
   \varphi(\mathbf{y}):=\inf_{i=1, \dots,N} \{ c(\mathbf{x}_i, \mathbf{y} ) -\bar{z_i} \}
\end{equation}
and note that it is an optimal transport map from $Q$ to $P$, set also the value $i(y)\in \{ 1,\dots,N\}$ where the infumum of \eqref{c-conj} is attained. As before, we can define their empirical counterparts 
\begin{equation}\label{c-con_n}
   \varphi_n(\mathbf{y}):=\inf_{i=1, \dots,N} \{ c(\mathbf{x}_i, \mathbf{y} ) -\hat{z}_i \},
\end{equation}
which is an optimal transport map from $Q$ to $P_n$, and $i_{n}(y)$ the index where the infimum of \eqref{c-con_n} is attained. Then we have that 
\begin{align}\label{sandwich}
\sqrt{n} (\hat{z}_{i_n(y)}-\bar{z}_{i_n(y)}) \leq \sqrt{n}(\varphi_n(\mathbf{y})-\varphi(\mathbf{y}))\leq  \sqrt{n} (\hat{z}_{i(y)}-\bar{z}_{i(y)}).
\end{align}
We can take supremums over $\mathbf{y}$ in both sides of \eqref{sandwich} and derive that
\begin{align*}
  \sqrt{n}\sup_{i=1, \dots,N}( \hat{z}_{i}-\bar{z}_{i})=\sqrt{n}\sup_{\mathbf{y}\in \mathcal{Y}}(\varphi_n(\mathbf{y})-\varphi(\mathbf{y})).
\end{align*}
By symmetry we have that 
\begin{align*}
  \sqrt{n}\sup_{i=1, \dots,N}| \hat{z}_{i}-\bar{z}_{i}|=\sqrt{n}\sup_{\mathbf{y}\in \mathcal{Y}}|\varphi_n(\mathbf{y})-\varphi(\mathbf{y})|,
\end{align*}
which implies the following corollary.
\begin{Corollary}\label{Corollary_pot}
Under the hypothesis of Theorem~\ref{Theo:potential}, for $\varphi$ and $\varphi_n$ defined in \eqref{c-conj} and \eqref{c-con_n}, we have that 
$$
\sqrt{n}\sup_{\mathbf{y}\in \mathcal{Y}}|\varphi_n(\mathbf{y})-\varphi(\mathbf{y})|\stackrel{w}\longrightarrow \sup_{i=1,\dots,N} |N_i|,
$$
where $(N_1,\dots, N_N)\sim N(\mathbf{0},\Sigma(\tilde{\mathbf{z}})).$
\end{Corollary}

We will conclude by  some comments on the assumptions made in this section.
\begin{enumerate}
\item Under the hypotheses of Theorem~\ref{Theo:potential}, the optimal potential is unique once we fix its value at a given point. Then Corollary \ref{Corollary_pot} provides a uniform confidence band for this optimal potential, namely, 
\begin{equation*}
   \left[ \varphi_n(\mathbf{y})\pm  \frac{\Delta_{\alpha}}{\sqrt{n}} \right]_{\mathbf{y}\in \mathcal{Y}},
\end{equation*}
where $\Delta_{\alpha}$ is  the $1-\alpha$ quantile of the limit distribution.
    \item Note that if we consider $\mathcal{R}=\R^d$ and the quadratic cost, then \eqref{Reg}, \eqref{Twist} and \eqref{QC} are obviously satisfied, by taking the function $\exp_j$ as the identity. Actually the map  $\mathbf{y}\mapsto|\mathbf{x}_j-\mathbf{y} |^2$ is $\mathcal{C}^{\infty}(\R^d)$ and $\mathbf{y}-\mathbf{x}_j$ is its derivative w.r.t. $\mathbf{y}$. Finally note that the function
    $$ \R^d\ni\mathbf{p}\mapsto |\mathbf{x}_i-\mathbf{p}|^2-|\mathbf{x}_j-\mathbf{p}|^2=|\mathbf{x}_i|^2-|\mathbf{x}_j|^2+\langle\mathbf{x}_j-\mathbf{x}_i,\mathbf{p}\rangle$$
    is linear in $\mathbf{p}$ and consequently quasi-convex.
\item Assumption~\eqref{PW} on the probability $Q$ has been widely studied in the literature for its implications in PDEs, see \cite{Acosta}. They proved that \eqref{PW} holds for a uniform distribution on a convex set $\Omega$. In \cite{Rathmair}, Lemma 1 claims that \eqref{PW} is equivalent to the bound of $\inf_{t\in \R} E(|f(Y)-t|)$, for every $f\in \mathcal{C}^1(\mathcal{Y})$. 
Let $Y\sim Q$ be such that there exists a $\mathcal{C}^1(\mathcal{Y})$ map $T$ satisfying the relation $T(U)=Y$, where $U$ follows a uniform distribution on a compact convex set $A$. Since $f\circ T\in \mathcal{C}^1(A)$,   by the powerful result of \cite{Acosta}, there exists $C_A>0$ such that 
\begin{align*}
       \inf_{t\in \R} E(|f(Y)-t|)&= \inf_{t\in \R} E(|f(T(U))-t|)\leq C_{A} E(|\nabla f(T(U))|\cdot ||T'(U)||_2)\\
    &\leq C_{A}\sup_{\mathbf{u}\in A}||T'(\mathbf{u})||_2E(|\nabla f(T(U))|),
\end{align*}
where $|| T'(U)||_2$ denotes the matrix operator norm. We conclude that, in such cases, \eqref{PW} holds. Note that the existence of this map relies on the well known existence of continuously differentiable optimal transport maps, which is treated by Caffarelli's theory. We refer to the most recent work \cite{cordero2019regularity} and references therein. However, as pointed out in \cite{JUN2019},  more general probabilities can satisfy that assumption such as radial functions on $\R^d$ with density 
\begin{equation*}
    \frac{p(|\mathbf{x}|)}{| \mathbf{x}|^{d-1}}, \ \ \text{for $|\mathbf{x}|\leq R$,  with $p=0$ in $[0,r] $ and concave in $[r,R]$. }
\end{equation*}
 Moreover the spherical uniform $\mathbb{U}_d$, used in \cite{Hallin2020DistributionAQ} to generalize the distribution function to general dimension, where we first choose the radium uniformly and then, independently, we choose a point in the sphere $\mathbf{S}_{d-1}$, also satisfies \eqref{PW}. This can be proved by using previous argument with the function $T(\mathbf{x})=\mathbf{x}| \mathbf{x}|^{d-1}$, which is continuously differentiable. But note that this probability measure does not satisfy \eqref{Hol}. We conjecture that Theorem \ref{Theo:potential} still holds in this case, but  some additional work should be done which is left as a future work. In the same way, the regularity of the transport can be obtained in the continuous case by a careful treatment of the Monge-Amp\'ere equation, see \cite{DELBARRIO2020104671}.
 \begin{figure}[h!]
    \centering
    \includegraphics[width=0.3\textwidth]{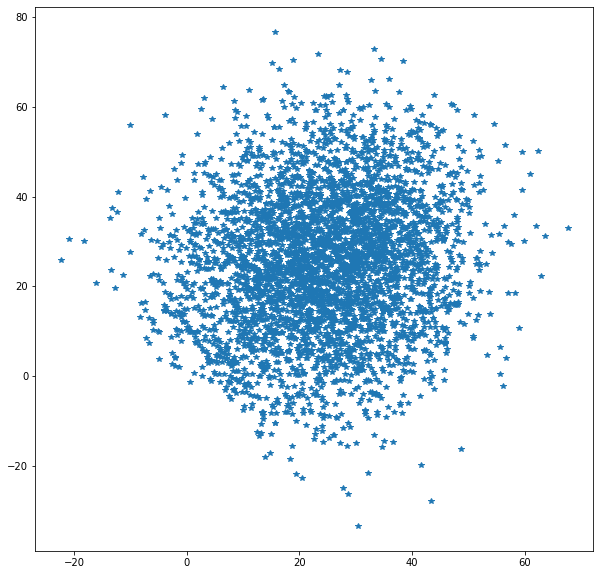}\ \ \ \ \ \ \ \ \ \ \ \ \ \ \ \ $ $ \ \ \ \
    \includegraphics[width=0.33\textwidth]{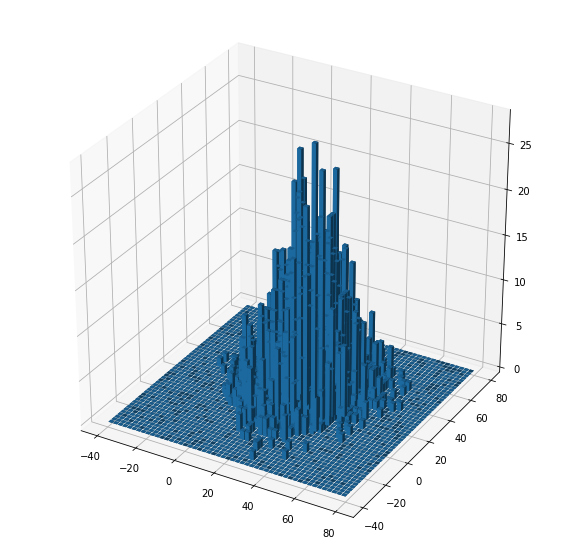}
    \caption{Bootstrap approximation of $N(\mathbf{0},\Sigma(\tilde{\mathbf{z}}))$ . Here $P$ is supported in tree points $P=\frac{1}{3}\left(\mathbf{x}_1+  \mathbf{x}_2+ \mathbf{x}_3\right)$, and $Q$  is uniform on $(0,1)^6.$ To approximate the uniform, we sample $50000$ i.i.d. points. Then we compute the empirical potentials $\hat{\mathbf{z}}$ for a sample of $ 4000 $ points  and the Bootstrap potentials $\hat{\mathbf{z}}^s$, for $s=1, \dots, 4000$.  Both---the empirical and the bootstrap---are projected to the space $\langle\mathbf{1}\rangle^{\perp}$. Since the space  $\langle\mathbf{1}\rangle^{\perp}$ is, in this case, $2-$dimensional, we can plot the $2D$ distribution of   $(\hat{\mathbf{z}}- \hat{\mathbf{z}}^s)\sqrt{4000}$  (left), and its $3D$ histogram (right).}
    \label{fig:mesh2}
\end{figure}\\
 \item The limit distribution described in \ref{Theo:potential} is not easy to derive, even knowing the exact probabilities $P$ and $Q$. But note that the limits are consequence of Corollary 2.2 in \cite{Huang2007CENTRALLT}, which used, in fact, a delta-method, for  differentiable functions in the classic sense. Hence a  bootstrap approximation can be used to approximate the limit distribution. The approximation will be consistent as  in~\cite{Santos}. In Figure~\ref{fig:mesh2} we compute such an approximation by using bootstrap where $P$ is supported on three  points in $\R^6$ and $Q$ is the uniform on $(0,1)^6$.

\end{enumerate}
\section*{Acknowledgements}
The authors would like to thank  Luis-Alberto Rodríguez  for showing us the paper \cite{Luis2020},  which is key for the proof of Theorem~\ref{Teoremaprinci}. The research of Eustasio del Barrio is partially supported by FEDER, Spanish Ministerio de Econom\'ia y Competitividad, grant MTM2017-86061-C2-1-P and Junta de Castilla y Le\'on, grants VA005P17 and VA002G18.
The research of Alberto Gonz\'alez-Sanz and Jean-Michel Loubes is partially supported by the AI Interdisciplinary Institute ANITI, which is funded by the French “Investing
for the Future – PIA3” program under the Grant agreement ANR-19-PI3A-0004.
\section{Appendix}

\subsection{Proofs of Lemmas}
\begin{proof}[Proof of Lemma \ref{Lemma:dualsemi}]
\mbox{}\\*
First, strong duality \eqref{dual} yields that 
\begin{align*}
\mathcal{T}_c(P,Q)=\sup_{(f,g)\in \Phi_c(P,Q)}\int f(\textbf{x}) dP(\textbf{x})+\int g(\textbf{y}) dQ(\textbf{y})= \sup_{(f,g)\in \Phi_c(P,Q)}\sum_{i=1}^N f(\mathbf{x_i})p_i+\int  g(\textbf{y})dQ(\mathbf{y}).
\end{align*}
Set $(z_1, \dots,z_N)=(f(\mathbf{x}_i), \dots, f(\mathbf{x}_N))$, then 
$
\mathcal{T}_c(P,Q)=\sup_{(\mathbf{z},g)}\sum_{i=1}^N z_ip_i+\int  g(\textbf{y})dQ(\mathbf{y}),
$
where the supremun is taken on the set $(\mathbf{z},g)$ such that $z_i+g(\textbf{y})\leq  c(\textbf{x}_i,\textbf{y})$ for all $i=1, \dots, N.$ Then  $g(\textbf{y})\leq  \inf_{i=1, \dots, N} \{ c(\mathbf{x}_i, \mathbf{y} ) -z_i \}$ and $\mathcal{T}_c(P,Q)=\sup_{\mathbf{z}\in \R^N} g_c(P,Q, \mathbf{z})$.\\

Let $\mathbf{z}^*=(z^*_1, \dots,z^*_N)\in\R^N$ be such that  $\mathcal{T}_c(P,Q)= g_c(P,Q, \mathbf{z}^*)$.
Denote as 
$ l=\arg\inf_i z_i^* $ and $ u=\arg\sup_i z_i^* $, which are different---otherwise the potentials are constant and we conclude that $K=0$. Therefore
\begin{align*}
 \mathcal{T}_c(P,Q)&\leq  \sum_{i=1}^N z^*_ip_i+\int  \{ c(\mathbf{y},\mathbf{x}_u) -z^*_u \}dQ(\mathbf{y})
 \\&\leq  (1-p_l)z^*_u+p_lz^*_l+\int  \{ c(\mathbf{y},\mathbf{x}_u) -z^*_u \}dQ(\mathbf{y})\\
& \leq - p_lz^*_u+p_lz^*_l+\int  c(\mathbf{y},\mathbf{x}_u) dQ(\mathbf{y})\\
& \leq  p_l(z^*_l-z^*_u)+\int  c(\mathbf{y},\mathbf{x}_u) dQ(\mathbf{y}),
\end{align*}
which implies $ p_l(z^*_u-z^*_l)\leq \int  c(\mathbf{y},\mathbf{x}_u) dQ(\mathbf{y})- \mathcal{T}_c(P,Q)$ and
\begin{align*}
 \sup_{i,j=1, \dots, N}|z^*_i-z^*_j |\leq \frac{1}{\inf_i p_i}\left( \sup_{i=1, \dots, N}\int  c(\mathbf{y},\mathbf{x}_i) dQ(\mathbf{y})-\mathcal{T}_c(P,Q) \right),
\end{align*}
since adding additive constant does not change $g_c(P,Q, \mathbf{z})$, then we conclude.
\end{proof}

\begin{Lemma}
\label{covering}
Under the assumptions of Theorem~\ref{Teoremaprinci},  the class   $ \mathcal{F}_c^K$ is $Q$-Donsker.
\end{Lemma}
\begin{proof}[Proof of Lemma \ref{covering}]
\mbox{}\\*

We use bracketing numbers, see Definition~2.1.6  in \cite{Vart_Well}.  Lemma~\ref{lemma8lipscit} implies that 
$${N}_{[]}(2\epsilon, \mathcal{F}_c^K, ||\cdot ||_{L^2(Q)})\leq {N}(\epsilon, \mathbb{B}_K(\mathbf{0}), |\cdot |). $$
Therefore, Lemma 4.14 in \cite{Massart2007ConcentrationIA} implies that 
\begin{equation}
    \label{entropycondition}
    \int_{0}^{\infty} \sqrt{\log\left({N}_{[]}(\epsilon, \mathcal{F}_c^K, ||\cdot ||_{L^2(Q)})\right)}d\epsilon<\infty.
\end{equation}
The envelope function of the class $\mathcal{F}_c^K$ can be taken as the function $F$ defined as $F(\mathbf{y})=  \sup_{i=1, \dots, m} c(\mathbf{x}_i, \mathbf{y} )+K.$ Note that $$\int F(\mathbf{y})^2dQ(\mathbf{y})\leq 2K+2\int \sup_{i=1, \dots, m} c(\mathbf{x}_i, \mathbf{y} )^2dQ(\mathbf{y})<\infty.$$ 
Using Theorem 3.7.38  in \cite{Gin2015MathematicalFO} we obtain the desired result.
\end{proof}

\bibliographystyle{plain}  
\bibliography{references}

\end{document}